\documentclass[dvipdfmx-if-dvi,autodetect-engine,ja=standard]{amsart}
\usepackage{amsmath,amsthm,amssymb}

\usepackage{mathtools}
\mathtoolsset{showonlyrefs=true} 

\numberwithin{equation}{section} 

\usepackage{amsfonts}
\usepackage{mathrsfs}
\usepackage{bbm}

\usepackage[setpagesize=false,colorlinks=true,linkcolor=blue,citecolor=blue]{hyperref}

\usepackage{mathtools} 
\mathtoolsset{showonlyrefs} 

\usepackage{here} 

\def\vect#1{\mbox{\boldmath $#1$}} 

\DeclareMathOperator{\loc}{loc}

\newtheorem{definition}{Definition}[section]

\newtheorem{theorem}[definition]{Theorem}
\newtheorem{lemma}[definition]{Lemma}

\newtheorem{remark}{Remark}[section]

\pagebreak

\title[FDM for wave equations with dynamic boundary conditions]{
Structure-preserving finite difference schemes
for nonlinear wave equations with dynamic boundary conditions}

\author[A. Umeda]{Akihiro Umeda}
\address[A. Umeda]{
Graduate School of Science and Engineering,
Ehime University,
Matsuyama, 790-8577, Japan
}

\author[Y. Wakasugi]{Yuta Wakasugi}
\address[Y. Wakasugi]{
Laboratory of Mathematics,
Graduate School of Advanced Science and Engineering,
Hiroshima University,
Higashi-Hiroshima, 739-8527, Japan
}
\email{wakasugi@hiroshima-u.ac.jp}

\author[S. Yoshikawa]{Shuji Yoshikawa}
\address[S. Yoshikawa]{
Division of Mathematical Sciences, 
Faculty of Science and Technology, Oita University, 
Oita, 870-1192, Japan }
\email{yoshikawa@oita-u.ac.jp}

\begin{document}
\begin{abstract}
In this article we discuss the numerical analysis for the finite difference scheme of the one-dimensional nonlinear 
wave equations with dynamic boundary conditions. From the viewpoint of the discrete variational derivative method 
we propose the derivation of the structure-preserving finite difference schemes of the problem which covers a variety of equations as 
widely as possible. Next, we focus our attention on the semilinear wave equation, and show the existence and uniqueness of solution for 
the scheme and error estimates with the help of the inherited energy structure. 
\end{abstract}

\maketitle
\tableofcontents

\section{Introduction}
Recently, a lot of interesting numerical analyses related with partial differential equations with the dynamic boundary condition are 
studied by many authors from various aspects (e.g. \cite{ch-pe}, \cite{4}, \cite{ko-lu2}, \cite{is-mi-pe}, \cite{na1}, \cite{na2}, \cite{fu-ok}, \cite{5} and 
reference therein). 
However, these results treat the parabolic equations such as the Cahn--Hilliard equation and the Allen--Cahn equation, and 
there seems to be no result which treats hyperbolic problems as far as authors know. 
In this article we discuss the numerical analysis from the viewpoint of a structure-preserving finite difference method for 1-D nonlinear wave equations 
with dynamic boundary conditions. 

The evolution equation with the dynamic boundary condition arises from the natural phenomena such as 
evolution with coating. 
The \emph{dynamic boundary condition} we call here means that the boundary condition itself is an evolution equation. 
The Cahn--Hilliard equation with dynamic boundary condition is extensively studied by many authors (see e.g. \cite{mi} and reference therein). 
On the other hand, there are a few results related with the wave equation with dynamic boundary condition in spite of the fact that the physical 
motivation is discussed in \cite{me}. 
In the article \cite{me}, to describe the dynamics of a piezoelectric stack actuator, 
the linear wave equation with a dynamic boundary condition is derived from the Hamilton principle. 
The corresponding simplified model to it is stated as follows:  
\begin{align*}
    \left\{ \begin{aligned}[2]
    &\frac{\partial^2 u}{\partial t^2} = \frac{\partial^2 u}{\partial x^2},
    &&(t,x) \in (0,T)\times (0,L),\\
    &\frac{\partial^2 u}{\partial t^2}(t,0)
    - \frac{\partial u}{\partial x}(t,0) = 0,
    &&t\in (0,T),\\
    &\frac{\partial^2 u}{\partial t^2}(t,L)
    + \frac{\partial u}{\partial x}(t,L) = 0,
    &&t\in (0,T).
    \end{aligned}\right.
\end{align*}
In this article, based on their idea in \cite{me}, 
we study the extended nonlinear problem including the problem. 
Namely, we introduce the derivation of the numerical schemes on the problems for a various kind of wave type equations by using the \emph{discrete variational derivative method} (DVDM) for wave type equations introduced in \cite{2}. 
DVDM is the systematic method to obtain finite difference schemes which inherit the energy conservation or decreasing (or possibly increasing) laws by defining the 
discrete version of variational derivative. For the precise information of the method we refer to \cite{4}. 
In this article, we apply the method to the hyperbolic problem with dynamic boundary conditions.  
As mentioned in Section 3,
as far as we concern the derivation of the scheme,
the procedure in this article can be applied to not only semilinear wave equations but also quasilinear wave equations. 

The \emph{energy method} is known as a classical method for differential equations to obtain various properties of their solutions by utilizing the energy structure of equations. 
One of examples is to construct a global-in-time solution by extending a local-in-time solution 
with the help of the energy conservation law as the a priori estimate. 
The method is easily extended to structure-preserving finite difference schemes thanks to their 
inherited energy structure. 
By the energy method we can prove a lot of things such as a construction of global-in-time solution 
for the scheme, a construction of local-in-time solution for the scheme and an error estimate 
between a strict solution for original problem and a solution for the scheme.  
For the energy method of the structure-preserving finite difference schemes, we refer to 
\cite{yo2} and \cite{3}. 
The method is applicable to the Cahn-Hilliard equation with dynamic boundary condition
(\cite{4}), and the procedure is more sophisticated in \cite{fu-ok} and \cite{5}. 
An example of application of the method to the wave type equation is given in \cite{ya-yo}. 
We give a numerical analysis such as the unique existence of solution for the scheme and the error estimate between a strict solution and the approximate solution to the semilinear wave equation, 
combining the ideas given in \cite{5} and \cite{ya-yo}. 

Let
$L>0$
and
$T>0$.
We consider
the nonlinear wave equation with
dynamic boundary conditions
\begin{align}\label{eq:ndw}
    \left\{ \begin{aligned}[2]
    &\frac{\partial^2 u}{\partial t^2} = - \frac{\delta G}{\delta u},
    &&(t,x) \in (0,T)\times (0,L),\\
    &\frac{\partial^2 u}{\partial t^2}(t,0)
    - \left. \frac{\partial G}{\partial u_x}\right|_{x=0} = 0,
    &&t\in (0,T),\\
    &\frac{\partial^2 u}{\partial t^2}(t,L)
    + \left. \frac{\partial G}{\partial u_x}\right|_{x=L} = 0,
    &&t\in (0,T).
    \end{aligned}\right.
\end{align}
Here,
$u$
denotes a real-valued unknown function
and
$G = G(u,u_x)$
is a density function of potential energy,
and we write
\begin{align}
    \frac{\delta G}{\delta u}
    = \frac{\partial G}{\partial u} - \frac{\partial}{\partial x} \frac{\partial G}{\partial u_x}.
\end{align}
We also used the notation such as $u_t:= \frac{\partial u}{\partial t}$ and $u_x := \frac{\partial u}{\partial x}$.

The equation \eqref{eq:ndw} is derived by
a variational method from the
energy of the form
\begin{align}
    J(u(t))
    &= \int_0^L
    \left\{ \frac{1}{2} u_t(t,x)^2 + G(u,u_x) \right\} \,dx
    + \frac{1}{2} u_t(t,0)^2 + \frac{1}{2} u_t(t,L)^2.
\end{align}
Indeed, we have
\begin{align}\label{eq:dtJ}
    \frac{d}{dt} J(u(t))
    &= \int_0^L \left\{ u_{tt} u_t + \frac{\partial}{\partial t} G(u,u_x) \right\} \,dx \\
    &\quad 
    + u_{tt}(t,0) u_t(t,0) + u_{tt}(t,L) u_t(t,L).
\end{align}
The integration by parts implies
\begin{align}
    \int_0^L \frac{\partial}{\partial t} G(u,u_x) \,dx
    &=
    \int_0^L \left(
        \frac{\partial G}{\partial u_x} u_{tx} + \frac{\partial G}{\partial u} u_t
    \right)\,dx \\
    &= \int_0^L \left(
            \frac{\partial G}{\partial u} - \frac{\partial}{\partial x} \frac{\partial G}{\partial u_x} 
            \right) u_t \,dx
        + \left[
            \frac{\partial G}{\partial u_x} u_t
        \right]_0^L \\
    &= \int_0^L \frac{\delta G}{\delta u} u_t \,dx
        + \left[
            \frac{\partial G}{\partial u_x} u_t
        \right]_0^L .
\end{align}
Therefore, the equation \eqref{eq:dtJ} becomes
\begin{align}
    \frac{d}{dt} J(u(t))
    &= \int_0^L \left( u_{tt} + \frac{\delta G}{\delta u} \right) u_t \,dx \\
    &\quad
    + \left[
            \frac{\partial G}{\partial u_x} u_t
        \right]_0^L
    + u_{tt}(t,0) u_t(t,0) + u_{tt}(t,L) u_t(t,L).
\end{align}
Thus, the equations \eqref{eq:ndw} are
derived so that the energy
$J(u(t))$
is preserved.

In particular, when the density function has the form
\begin{align}
    G(u,u_x) = \frac{1}{2} u_x^2 + F(u),
\end{align}
and the energy is defined by
\begin{align}\label{eq:J:sdw}
    J(u(t))
    &=  \int_0^L
    \left(
    \frac{1}{2} \left( u_t(t,x)^2 + u_x(t,x)^2 \right) + F(u) \right) \,dx \\
    &\quad + \frac{1}{2}u_t(t,0)^2 + \frac{1}{2}u_t(t,L)^2,
\end{align}
the derived equation is semilinear, and the
corresponding problem to \eqref{eq:ndw} is given by
\begin{align}\label{eq:sdw}
    \left\{ \begin{aligned}[2]
    &\frac{\partial^2 u}{\partial t^2} = \frac{\partial^2 u}{\partial x^2} - F'(u),
    &&(t,x) \in (0,T)\times (0,L),\\
    &\frac{\partial^2 u}{\partial t^2}(t,0)
    - \frac{\partial u}{\partial x}(t,0) = 0,
    &&t\in (0,T),\\
    &\frac{\partial^2 u}{\partial t^2}(t,L)
    + \frac{\partial u}{\partial x}(t,L) = 0,
    &&t\in (0,T).
    \end{aligned}\right.
\end{align}
As also mentioned above, in this article, we will derive
a difference scheme for the problem \eqref{eq:sdw}
which inherits the energy conservation property
for the discrete energy corresponding to \eqref{eq:J:sdw}.
Moreover, we discuss the existence of the solution
for the derived difference scheme,
and estimates of errors between
the strict and approximate solutions.
We also give an energy conservative difference scheme
corresponding to the general nonlinear wave equation
\eqref{eq:ndw}.

This paper is organized as follows.
In the next section, we prepare
the notations for difference operators and summations
used in this paper.
The notion and properties of difference quotient
are also given.
In Section 3, we give
the difference schemes corresponding to the 
semilinear problem \eqref{eq:sdw}
and general nonlinear problem \eqref{eq:ndw}.
We also give some typical examples.
Section 4 is devoted to
the existence of the solution for
the difference scheme of the semilinear problem
derived in Section 3.
In section 5, we discuss about
estimates of errors for 
the difference scheme in the semilinear case.
Finally, in Section 6,
we show some numerical computations.

\section{Preliminaries}
\subsection{Notations}
Let
$L, T > 0$,
and let
$N$ and $K$ be positive integers.
Let
$\Delta t = T/N$
and
$\Delta x = L/K$.

For a family of vectors
$\{ f_k^{(n)} \}_{k=0}^K \ (n=0,1,\ldots,N)$,
we define the shift operators by
\begin{alignat}{3}
    s_n^+ f_k^{(n)}
    &= f_k^{(n+1)},
    &\quad
    s_n^- f_k^{(n)}
    &= f_k^{(n-1)},
    &\quad
    s_n^{\langle 1 \rangle} f_k^{(n)}
    &= \frac{1}{2}\left( f_k^{(n+1)}+f_k^{(n-1)}\right),\\
    s_k^+ f_k^{(n)}
    &= f_{k+1}^{(n)},
    &\quad
    s_k^- f_k^{(n)}
    &= f_{k-1}^{(n)},
    &\quad
    s_k^{\langle 1 \rangle} f_k^{(n)}
    &= \frac{1}{2}\left( f_{k+1}^{(n)} + f_{k-1}^{(n)} \right),
\end{alignat}
and the mean operators by
\begin{alignat}{2}
    \mu_n^+ f_k^{(n)}
    &= \left( \frac{s_n^+ + 1}{2} \right) f_k^{(n)},
    &\quad
    \mu_n^- f_k^{(n)}
    &=\left( \frac{s_n^- + 1}{2} \right) f_k^{(n)} ,\\
    \mu_k^+ f_k^{(n)}
    &= \left( \frac{s_k^+ + 1}{2} \right) f_k^{(n)},
    &\quad
    \mu_k^- f_k^{(n)}
    &= \left( \frac{s_k^- + 1}{2} \right) f_k^{(n)}.
\end{alignat}
Define the first-order difference operators by
\begin{alignat}{3}
    \delta_n^+ f_{k}^{(n)} &= \frac{f_k^{(n+1)}-f_k^{(n)}}{\Delta t},
    &\quad \delta_n^- f_{k}^{(n)} &= \frac{f_k^{(n)}-f_k^{(n-1)}}{\Delta t},
    &\quad \delta_n^{\langle 1\rangle} f_k^{(n)} &=
    \frac{f_k^{(n+1)}-f_k^{(n-1)}}{2\Delta t},\\
    \delta_k^+ f_{k}^{(n)} &= \frac{f_{k+1}^{(n)}-f_k^{(n)}}{\Delta x},
    &\quad \delta_k^- f_{k}^{(n)} &= \frac{f_k^{(n)}-f_{k-1}^{(n)}}{\Delta x},
    &\quad \delta_k^{\langle 1\rangle} f_k^{(n)} &=
    \frac{f_{k+1}^{(n)}-f_{k-1}^{(n)}}{2\Delta x},
\end{alignat}
and the second-order difference operators by
\begin{align}
    \delta_n^{\langle 2 \rangle} f_k^{(n)}
    = \frac{f_{k}^{(n+1)}-2f_k^{(n)}+f_k^{(n-1)}}{(\Delta t)^2},\quad
    \delta_k^{\langle 2 \rangle} f_k^{(n)}
    = \frac{f_{k+1}^{(n)} - 2 f_k^{(n)} + f_{k-1}^{(n)}}{(\Delta x)^2}.
\end{align}
We will use a summation defined by
\begin{align}
    \sum_{k=0}^K{}'' f_k^{(n)} \Delta x
    = \frac{1}{2} f_0^{(n)}\Delta x
    + \sum_{k=1}^{K-1} f_k^{(n)} \Delta x
    + \frac{1}{2} f_K^{(n)} \Delta x,
\end{align}
which follows the trapezoidal rule. 
We note that it is also expressed as
\begin{align}
     \sum_{k=0}^K{}'' f_k^{(n)} \Delta x
    = \frac{1}{2} \sum_{k=1}^K f_{k}^{(n)} \Delta x
    + \frac{1}{2} \sum_{k=0}^{K-1} f_{k}^{(n)} \Delta x.
\end{align}
The summation by parts formula is given by
\begin{equation}\label{sbp}
    \sum_{k=0}^{K-1} f_k^{(n)} \delta_k^+ g_k^{(n)} \Delta x
    + \sum_{k=0}^K{}'' (\delta_k^- f_k^{(n)}) g_k^{(n)} \Delta x
    =
    \left[ \left( \mu_k^- f_k^{(n)} \right) g_k^{(n)}
    \right]_0^K, 
\end{equation}
which can be found in e.g. \cite{5}. 

For
$\vect{U} = \{ U_k \}_{k=0}^K \in \mathbb{R}^{K+1}$,
we define the norms 
\begin{align}
    \| \vect{U} \|_{\infty}
    &:= \max_{0\le k \le K} |U_k|,\\
    \| \vect{U} \|_2
    &:= \left( \sum_{k=0}^{K}{}'' |U_k|^2 \Delta x \right)^{1/2},\\
    \| \vect{U} \|_{H^1}
    &:= \left(
        \sum_{k=0}^K{}'' |U_k|^2 \Delta x
        + \sum_{k=0}^{K-1} |\delta_k^+ U_k|^2 \Delta x
    \right)^{1/2}.
\end{align}
We will also denote the semi-norm $(\sum_{k=0}^{K-1} |\delta_k^+ U_k|^2 \Delta x)^{1/2}$ by $\| D \vect{U} \|$. 
\begin{lemma}[Discrete Sobolev lemma (\cite{yo2}, Proposition 2.2)]\label{lem:Sobolev}
For $\vect{U} = \{ U_k \}_{k=0}^{K} \in \mathbb{R}^{K+1}$,
we have
\begin{align}
    \| \vect{U} \|_{\infty}
    \le C_S
    \| \vect{U} \|_{H^1}, 
\end{align}
where $C_S := \sqrt{\frac{\sqrt{1+4L^2}+1}{2L}}$. 
\end{lemma}

\subsection{Notion and some properties of difference quotient}
We shall use two types of
difference quotient.
For a $C^1$ function
$F: \mathbb{R} \to \mathbb{R}$,
we define
the two-point difference quotient
by
\begin{align}
    \frac{dF}{d(a,b)}
    :=
    \begin{dcases}
    \frac{F(a) - F(b)}{a-b}
    &(a\neq b),\\
    F'(a)
    &(a=b).
    \end{dcases}
\end{align}
Also, following \cite{1, 2},
we use the notation of four-point difference quotient.
For a $C^1$ function
$f : \mathbb{R}^2 \to \mathbb{R}$
and for 
$a,b,c,d \in \mathbb{R}$,
we define
\begin{align}
    \frac{d(f,\loc)}{d(a,b:c,d)}
    :=
    \begin{dcases}
    \frac{f(a,b)-f(c,d)}{\loc (a,b)-\loc (c,d)}
    &(\loc (a,b) \neq \loc (c,d)),\\
   \left. \frac{d f(x,x)}{dx} \right|_{x=\loc (a,b)}
   &(\loc (a,b) = \loc (c,d)),
    \end{dcases}
\end{align}
where 
\begin{align}
    \loc (a,b) := \frac{1}{2}(a+b).
\end{align}

In the mathematical analysis for the numerical scheme in Sections 4 and 5, 
we impose some assumption which makes us regard the four-point finite difference 
quotient as two-point one, for simplicity. 
We can utilize the following useful properties for the two-point difference quotient introduced in e.g. \cite{yo2}.
Let $\Omega = [0,L]$.
For $F \in C^1(\Omega)$ and $\xi, \eta \in \Omega$,
we have
\begin{equation}\label{mvt}
\inf_{\xi \in \Omega} F'(\xi) \leq \frac{d F}{d(\xi, \eta)}
	\leq \sup_{\xi \in \Omega} F'(\xi),
\end{equation}
which is easily seen from the mean value theorem. 
As we will see later, a subtraction between difference quotients often appears in the proofs of existence and error estimate.  
To treat it easily and systematically, 
we define the \textit{averaged second order difference quotient} $\overline{F}''(\xi,\tilde{\xi};\eta,\tilde{\eta})$ for $F \in C^2$ by    
\begin{equation*}
\begin{split}
\overline{F}''(\xi,\tilde{\xi};\eta,\tilde{\eta}) &:= \frac{d}{d (\xi,\tilde{\xi})}
	 \left( \frac{d F}{d(\cdot, \eta)} + \frac{d F}{d(\cdot, \tilde{\eta})} \right)  \\ 
	&= 
	\begin{dcases} \frac{1}{\xi-\tilde{\xi}} \left\{  
	\left( \frac{d F}{d(\xi, \eta)} + \frac{d F}{d(\xi, \tilde{\eta})} \right) 
		-  \left( \frac{d F}{d(\tilde{\xi}, \eta)} + \frac{d F}{d(\tilde{\xi}, \tilde{\eta})} \right) \right\}
		&(\xi \neq \tilde{\xi}), \\ 
	 \partial_{\xi} 
		\left( \frac{d F}{d(\xi, \eta)} + \frac{d F}{d(\xi, \tilde{\eta})} \right) \big|_{\xi=\tilde{\xi}}
	&(\xi = \tilde{\xi}), 
	\end{dcases}
\end{split}
\end{equation*}
which is a kind of second order difference quotient. 
For $F \in C^2(\Omega)$ and $\xi$, $\tilde{\xi}$, $\eta$, $\tilde{\eta} \in \Omega$ it holds that 
\begin{equation}\label{eq:2nd-dq}
\frac{d F}{d(\xi, \eta)} - \frac{d F}{d(\tilde{\xi}, \tilde{\eta})} 
	= \frac{1}{2} \overline{F}''(\xi,\tilde{\xi};\eta,\tilde{\eta}) \cdot (\xi-\tilde{\xi}) 
	+  \frac{1}{2} \overline{F}''(\eta,\tilde{\eta};\xi,\tilde{\xi}) \cdot (\eta-\tilde{\eta}). 
\end{equation}
If $F \in C^2(\Omega)$, then $\overline{F}'' \in C(\Omega)$. Moreover, 
for any $\xi$, $\tilde{\xi}$, $\eta$, $\tilde{\eta} \in \Omega$ it holds that 
\begin{equation}\label{mvt2}
\inf_{\xi \in \Omega} F''(\xi) \leq \overline{F}''(\xi,\tilde{\xi};\eta,\tilde{\eta}) 
	\leq \sup_{\xi \in \Omega} F''(\xi).  
\end{equation}
For the proof we refer to \cite{yo2}. 
Let us denote the vector form of the difference quotient by
$\dfrac{d F}{d(\vect{U}, \vect{V})} := \left\{ \dfrac{d F}{d(U_k, V_k)} \right\}_{k=0}^K$.
By using these properties we obtain the following lemma. 
\begin{lemma}\label{prop-f}
Suppose that
$F\in C^2$. 
Let
$\vect{U}$, $\vect{V} \in H^1$
and set
$R := \max \{ \| \vect{U} \|_{H^1},  \| \vect{V} \|_{H^1} \}$. 
Let us define $C_{F}(\rho) := \max_{|\xi| \leq \rho} |F(\xi)|$. 
\begin{enumerate}
\item It holds that 
\begin{equation}\label{eq:1st-dq-l2}
\left\| \frac{d F}{d(\vect{U}, \vect{V})} \right\|_2 \leq \min \left\{  C_{F'}(C_S R) \sqrt{L}, \  |F'(0)| \sqrt{L} + C_{F''}(C_S R) (\| \vect{U} \|_2 + \| \vect{V} \|_2) \right\}. 
\end{equation}
where $C_S$ is defined in Lemma \ref{lem:Sobolev}. 
\item It holds that 
\begin{equation}\label{eq:2nd-dq-l2}
\left\| \frac{d F}{d(\vect{U}, \vect{V})} - \frac{d F}{ d(\tilde{\vect{U}}, \tilde{\vect{V}}) }  \right\|_2 \leq
\frac{1}{2} C_{F''}(C_S R) (\| \vect{U} - \tilde{\vect{U}} \|_2 + \| \vect{V} - \tilde{\vect{V}} \|_2)
\end{equation}
\end{enumerate}
\end{lemma}
\begin{proof}
Observe that
$\| \vect{U} \|_{\infty}, \| \vect{V} \|_{\infty} \leq C_S R$
from the Sobolev inequality. 
From \eqref{mvt} we have 
\[
\left\| \frac{d F}{d(\vect{U}, \vect{V})} \right\|_2 \leq C_{F'}(C_S R) \sqrt{L}. 
\]
On the other hand, since from \eqref{eq:2nd-dq} 
\[
\frac{d F}{d(U_k, V_k)} - F'(0) 
	= \frac{1}{2} \overline{F}''(U_k,0;V_k,0) \cdot U_k 
	+  \frac{1}{2} \overline{F}''(V_k,0;U_k,0) \cdot V_k, 
\]
we obtain 
\[
\left\| \frac{d F}{d(\vect{U}, \vect{V})} \right\|_2 \leq |F'(0)| \sqrt{L} + C_{F''}(C_S R) (\| \vect{U} \|_2 + \| \vect{V} \|_2),
\] 
which proves (1). 
The second estimate (2) is an easy consequence from \eqref{eq:2nd-dq} and \eqref{mvt2}. 
\end{proof}

\section{Energy conservation schemes}
\subsection{Semilinear case}
We consider the energy conservation scheme
for the semilinear wave equation \eqref{eq:sdw}.
Let
$F : \mathbb{R}^2 \to \mathbb{R}$
be a
$C^1$
function.
For
$\vect{U}^{(n)} = \{ U_k^{(n)} \}_{k=-1}^{K+1} \in \mathbb{R}^{K+3} \ (n=0,1,\ldots,N)$,
we define
discrete energy density functions
$G_{\mathrm{d},k}^+ (\vect{U}^{(n+1)}, \vect{U}^{(n)})$
and
$G_{\mathrm{d},k}^- (\vect{U}^{(n+1)}, \vect{U}^{(n)})$
by the form
\begin{align}\label{eq:semi:Gdk:+}
    G_{\mathrm{d},k}^+ (\vect{U}^{(n+1)}, \vect{U}^{(n)})
    &= 
    \frac{1}{2}
    \left(
        \frac{(\delta_k^+ U_k^{(n+1)})^2 + (\delta_k^+ U_k^{(n)})^2}{2}
    \right)
    + F(U_k^{(n+1)}, U_k^{(n)}),\\
\label{eq:semi:Gdk:-}
     G_{\mathrm{d},k}^- (\vect{U}^{(n+1)}, \vect{U}^{(n)})
    &= 
    \frac{1}{2} s_k^- 
    \left(
        \frac{(\delta_k^+ U_k^{(n+1)})^2 + (\delta_k^+ U_k^{(n)})^2}{2}
    \right)
    + F(U_k^{(n+1)}, U_k^{(n)}).
\end{align}
Moreover, we define a discrete energy
$J_{\mathrm{d}}(\vect{U}^{(n+1)}, \vect{U}^{(n)})$
by
\begin{align}\label{eq:semi:J}
    J_{\mathrm{d}}(\vect{U}^{(n+1)}, \vect{U}^{(n)})
    &= \sum_{k=0}^K{}'' \frac{1}{2} (\delta_n^+ U_k^{(n)})^2 \Delta x \\
    &\quad + 
    \frac{1}{2} \sum_{k=0}^{K-1}
    G_{\mathrm{d},k}^+(\vect{U}^{(n+1)}, \vect{U}^{(n)}) \Delta x
    + \frac{1}{2} \sum_{k=1}^K
    G_{\mathrm{d},k}^- (\vect{U}^{(n+1)}, \vect{U}^{(n)}) \Delta x \\
    &\quad
    + \frac{1}{2} (\delta_n^+ U_0^{(n)})^2
    + \frac{1}{2} (\delta_n^+ U_K^{(n)})^2 \\
    &=
    \sum_{k=0}^K{}'' \frac{1}{2} (\delta_n^+ U_k^{(n)})^2 \Delta x \\
    &\quad +
    \sum_{k=0}^{K-1} 
    \frac{1}{2}
    \left(
        \frac{(\delta_k^+ U_k^{(n+1)})^2 + (\delta_k^+ U_k^{(n)})^2}{2}
    \right)
    \Delta x
    + \sum_{k=0}^K{}'' F(U_k^{(n+1)}, U_k^{(n)}) \Delta x \\
    &\quad
    + \frac{1}{2} (\delta_n^+ U_0^{(n)})^2
    + \frac{1}{2} (\delta_n^+ U_K^{(n)})^2.
\end{align}
Our difference scheme which
preserves the energy
$J_{\mathrm{d}}(\vect{U}^{(n+1)}, \vect{U}^{(n)})$
is the following.
\begin{align}\label{eq:scheme:semi}
    \left\{
    \begin{alignedat}{1}
    &\delta_n^{\langle 2\rangle} U_k^{(n)}
    = \delta_k^{\langle 2 \rangle}
    \left(
        \frac{U_k^{(n+1)}+U_k^{(n-1)}}{2}
    \right)
    -
        \frac{d(F,\loc)}{d(U_k^{(n+1)},U_k^{(n)} : U_k^{(n)}, U_k^{(n-1)})}, \\
    &\qquad \qquad \qquad \qquad \qquad \qquad \qquad \qquad 
        (k=0,\ldots,K, n=1,\ldots,N-1),\\
    &\delta_n^{\langle 2\rangle} U_0^{(n)}
        - \delta_k^{\langle 1 \rangle}
        \left(
            \frac{U_0^{(n+1)}+U_0^{(n-1)}}{2}
        \right)
        = 0,
        \qquad (n=1,\ldots,N-1),\\
    &\delta_n^{\langle 2\rangle} U_K^{(n)}
       + \delta_k^{\langle 1 \rangle}
        \left(
            \frac{U_K^{(n+1)}+U_K^{(n-1)}}{2}
        \right)
        = 0,
        \qquad (n=1,\ldots,N-1).
    \end{alignedat}
    \right.
\end{align}
The proposed scheme above preserves the energy:
\begin{theorem}[Energy conservation law]\label{thm:energy:conserv:semi}
If
$\vect{U}^{(n)} = \{ U_k^{(n)} \}_{k=-1}^{K+1} \in \mathbb{R}^{K+3} \ (n=0,1,\ldots,N)$
satisfies the scheme \eqref{eq:scheme:semi},
then
$\delta_n^- J_{\mathrm{d}}(\vect{U}^{(n+1)}, \vect{U}^{(n)}) = 0$
for all
$n =1,\ldots,N-1$.
\end{theorem}
\begin{proof}
We calculate
\begin{align}
    \delta_n^- J_{\mathrm{d}}(\vect{U}^{(n+1)}, \vect{U}^{(n)})
    &=
    \sum_{k=0}^K{}'' \frac{1}{2} \delta_n^- (\delta_n^+ U_k^{(n)})^2 \Delta x \\
    &\quad +
    \sum_{k=0}^{K-1} \frac{1}{2} \delta_n^-
    \left(
        \frac{(\delta_k^+ U_k^{(n+1)})^2 + (\delta_k^+ U_k^{(n)})^2}{2}
    \right)
    \Delta x \\
    &\quad + \sum_{k=0}^K{}'' \delta_n^- F(U_k^{(n+1)},U_k^{(n)}) \Delta x \\
    &\quad + 
    \frac{1}{2} \delta_n^- (\delta_n^+ U_0^{(n)})^2
    + \frac{1}{2} \delta_n^- (\delta_n^+ U_K^{(n)})^2.
\end{align}
We first note that
$\delta_n^- (\delta_n^+ U_k^{(n)})^2
= 2 \delta_n^{\langle 2\rangle}U_k^{(n)} \delta_n^{\langle 1 \rangle} U_k^{(n)}$.
The last two terms are also computed similarly.
Moreover, by the summation by parts \eqref{sbp},
the second term of the right-hand side is
calculated as
\begin{align}
    &\sum_{k=0}^{K-1} \frac{1}{2} \delta_n^-
    \left(
        \frac{(\delta_k^+ U_k^{(n+1)})^2 + (\delta_k^+ U_k^{(n)})^2}{2}
    \right)
    \Delta x \\
    &=
    \sum_{k=0}^{K-1} \frac{1}{2}
    \frac{(\delta_k^+ U_k^{(n+1)})^2 - (\delta_k^+ U_k^{(n-1)})^2}{2\Delta t} \Delta x \\
    &=
    \sum_{k=0}^{K-1} 
    \left(
    \frac{\delta_k^+ U_k^{(n+1)}+\delta_k^+ U_k^{(n-1)}}{2}
    \right)
    \delta_k^+ \delta_n^{\langle 1 \rangle} U_k^{(n)} \\
    &=
    - \sum_{k=0}^{K}{}''
        \delta_k^{\langle 2 \rangle}
        \left(
    \frac{U_k^{(n+1)} + U_k^{(n-1)}}{2}
    \right)
    \delta_n^{\langle 1 \rangle} U_k^{(n)} \\
    &\quad
    + \left[
    \mu_k^- \left(
    \frac{\delta_k^+ U_k^{(n+1)}+\delta_k^+ U_k^{(n-1)}}{2}
    \right)
    \delta_n^{\langle 1 \rangle} U_k^{(n)}
    \right]_{0}^K \\
    &=
    - \sum_{k=0}^{K}{}''
        \delta_k^{\langle 2 \rangle}
        \left(
    \frac{U_k^{(n+1)} + U_k^{(n-1)}}{2}
    \right)
    \delta_n^{\langle 1 \rangle} U_k^{(n)} \\
    &\quad
    + \delta_k^{\langle 1 \rangle}
    \left(
    \frac{U_K^{(n+1)} + U_K^{(n-1)}}{2}
    \right)
    \delta_n^{\langle 1 \rangle} U_K^{(n)}
    - \delta_k^{\langle 1 \rangle}
    \left(
    \frac{U_0^{(n+1)} + U_0^{(n-1)}}{2}
    \right)
    \delta_n^{\langle 1 \rangle} U_0^{(n)}.
\end{align}
We also compute
\begin{align}
    \sum_{k=0}^K{}'' \delta_n^- F(U_k^{(n+1)}, U_k^{(n)})\Delta x
    &= \sum_{k=0}^{K}{}''
    \frac{d(F,\loc)}{d(U_k^{(n+1)}, U_k^{(n)} : U_k^{(n)}, U_k^{(n-1)})} \delta_n^{\langle 1 \rangle} U_k^{(n)} \Delta x.
\end{align}
From them and the equations \eqref{eq:scheme:semi},
we conclude
\begin{align}
    &\delta_n^- J_{\mathrm{d}} (\vect{U}^{(n+1)}, \vect{U}^{(n)}) \\
    &=
    \sum_{k=0}^K{}''
    \left\{ 
    \delta_n^{\langle 2 \rangle} U_k^{(n)}
    - \delta_k^{\langle 2 \rangle}
        \left(
    \frac{U_k^{(n+1)} + U_k^{(n-1)}}{2}
    \right)
    \right. \\
    &\qquad \qquad \qquad \left.
    + \frac{d(F,\loc)}{d(U_k^{(n+1)}, U_k^{(n)} : U_k^{(n)}, U_k^{(n-1)})} 
         \right\}
         \delta_n^{\langle 1 \rangle} U_k^{(n)} \Delta x \\
    &\quad
    + \left\{
    \delta_n^{\langle 2 \rangle} U_0^{(n)}
    - \delta_k^{\langle 1 \rangle}
    \left(
    \frac{U_0^{(n+1)} + U_0^{(n-1)}}{2}
    \right)
    \right\}
    \delta_n^{\langle 1 \rangle} U_0^{(n)} \\
    &\quad
    + \left\{
    \delta_n^{\langle 2 \rangle} U_K^{(n)}
    + \delta_k^{\langle 1 \rangle}
    \left(
    \frac{U_K^{(n+1)} + U_K^{(n-1)}}{2}
    \right)
    \right\}
    \delta_n^{\langle 1 \rangle} U_K^{(n)} \\
    &= 0,
\end{align}
which completes the proof.
\end{proof}

\subsection{Examples}
\subsubsection{Semilinear wave equation}
We consider the semilinear wave equation
with a defocusing nonlinearity
$-u^3$.
\begin{align}\label{eq:sdw:example:-u^3}
    \left\{ \begin{aligned}[2]
    &\frac{\partial^2 u}{\partial t^2} = \frac{\partial^2 u}{\partial x^2} - u^3,
    &&(t,x) \in (0,T)\times (0,L),\\
    &\frac{\partial^2 u}{\partial t^2}(t,0)
    - \frac{\partial u}{\partial x}(t,0) = 0,
    &&t\in (0,T),\\
    &\frac{\partial^2 u}{\partial t^2}(t,L)
    + \frac{\partial u}{\partial x}(t,L) = 0,
    &&t\in (0,T).
    \end{aligned}\right.
\end{align}
The corresponding energy is given by
\begin{align}
    J(u(t)) = \int_0^L \left( 
    \frac{1}{2}
    \left( u_t(t,x)^2 + u_x(t,x)^2 \right) + \frac{u(t,x)^4}{4}
    \right) \,dx
    + \frac{1}{2}u_t(t,0)^2 + \frac{1}{2}u_t(t,L)^2
\end{align}
We define a function
$F: \mathbb{R}^2 \to \mathbb{R}$
by
\begin{align}
    F(a,b) := \frac{1}{2} \left( \frac{a^4}{4} + \frac{b^4}{4} \right)
\end{align}
and the discrete energy by
\begin{align}
    J_{\mathrm{d}} (\vect{U}^{(n+1)}, \vect{U}^{(n)})
    &=
    \sum_{k=0}^K{}'' \frac{1}{2} (\delta_n^+ U_k^{(n)})^2 \Delta x \\
    &\quad +
    \sum_{k=0}^{K-1} 
    \frac{1}{2}
    \left(
        \frac{(\delta_k^+ U_k^{(n+1)})^2 + (\delta_k^+ U_k^{(n)})^2}{2}
    \right)
    \Delta x \\
    &\quad + \sum_{k=0}^K{}'' \frac{1}{8}\left( (U_k^{(n+1)})^4 + (U_k^{(n)})^4 \right) \Delta x \\
    &\quad
    + \frac{1}{2} (\delta_n^+ U_0^{(n)})^2
    + \frac{1}{2} (\delta_n^+ U_K^{(n)})^2.
\end{align}
Then, we easily see
\begin{align}
    &\frac{d(F,\loc)}{d(U_k^{(n+1)}, U_k^{(n)} : U_k^{(n)}, U_k^{(n-1)})} \\
    &=
    \frac{1}{4}
    \left(
    (U_k^{(n+1)})^3 + (U_k^{(n+1)})^2 U_k^{(n-1)} + U_k^{(n+1)} (U_k^{(n-1)})^2 + (U_k^{(n-1)})^3
    \right),
\end{align}
and hence, the energy conservation scheme is given by
\begin{align}\label{eq:scheme:-u^3:sec3}
    \left\{
    \begin{alignedat}{1}
    &\delta_n^{\langle 2\rangle} U_k^{(n)}
    = \delta_k^{\langle 2 \rangle}
    \left(
        \frac{U_k^{(n+1)}+U_k^{(n-1)}}{2}
    \right) \\
    &\qquad \qquad - \frac{1}{4} \left(
    (U_k^{(n+1)})^3 + (U_k^{(n+1)})^2 U_k^{(n-1)}
    + U_k^{(n+1)} (U_k^{(n-1)})^2 + (U_k^{(n-1)})^3
    \right),\\
    &\delta_n^{\langle 2\rangle} U_0^{(n)}
        - \delta_k^{\langle 1 \rangle}
        \left(
            \frac{U_0^{(n+1)}+U_0^{(n-1)}}{2}
        \right)
        = 0,\\
    &\delta_n^{\langle 2\rangle} U_K^{(n)}
       + \delta_k^{\langle 1 \rangle}
        \left(
            \frac{U_K^{(n+1)}+U_K^{(n-1)}}{2}
        \right)
        = 0.
    \end{alignedat}
    \right.
\end{align}

\subsubsection{Sine-Gordon equation}
We consider the sine-Gordon equation
\begin{align}\label{eq:sdw:example:sineGordon}
    \left\{ \begin{aligned}[2]
    &\frac{\partial^2 u}{\partial t^2} = \frac{\partial^2 u}{\partial x^2} - \sin (u),
    &&(t,x) \in (0,T)\times (0,L),\\
    &\frac{\partial^2 u}{\partial t^2}(t,0)
    - \frac{\partial u}{\partial x}(t,0) = 0,
    &&t\in (0,T),\\
    &\frac{\partial^2 u}{\partial t^2}(t,L)
    + \frac{\partial u}{\partial x}(t,L) = 0,
    &&t\in (0,T).
    \end{aligned}\right.
\end{align}
The corresponding energy is defined by
\begin{align}
    J(u(t))
    &=
    \int_0^L \left(
    \frac{1}{2} \left( u_t(t,x)^2 + u_x(t,x)^2 \right)
    + (1-\cos (u(t,x)))
    \right)\,dx \\
    &\quad
    + \frac{1}{2}u_t(t,0)^2 + \frac{1}{2}u_t(t,L)^2.
\end{align}
We define a function
$F: \mathbb{R}^2 \to \mathbb{R}$
by
\begin{align}
    F(a,b) := \frac{1}{2} \left( 2- \cos a - \cos b \right)
\end{align}
and the discrete energy by
\begin{align}
    J_{\mathrm{d}} (\vect{U}^{(n+1)}, \vect{U}^{(n)})
    &=
    \sum_{k=0}^K{}'' \frac{1}{2} (\delta_n^+ U_k^{(n)})^2 \Delta x \\
    &\quad +
    \sum_{k=0}^{K-1} 
    \frac{1}{2}
    \left(
        \frac{(\delta_k^+ U_k^{(n+1)})^2 + (\delta_k^+ U_k^{(n)})^2}{2}
    \right)
    \Delta x \\
    &\quad + \sum_{k=0}^K{}'' \frac{1}{2}\left(2- \cos (U_k^{(n+1)}) - \cos (U_k^{(n)}) \right) \Delta x \\
    &\quad
    + \frac{1}{2} (\delta_n^+ U_0^{(n)})^2
    + \frac{1}{2} (\delta_n^+ U_K^{(n)})^2.
\end{align}
Then, by a straightforward calculation,
we have
\begin{align}
    \frac{d(F,\loc)}{d(U_k^{(n+1)}, U_k^{(n)} : U_k^{(n)}, U_k^{(n-1)})}
    &=
    \mathrm{sinc} \left( \frac{U_k^{(n+1)}-U_k^{(n-1)}}{2} \right)
    \sin \left( \frac{U_k^{(n+1)}+U_k^{(n-1)}}{2} \right),
\end{align}
where
$\mathrm{sinc}(x) := \frac{\sin x}{x}$
for $x\neq 0$ and
$\mathrm{sinc}(0) := 1$.
Hence, the energy conservation scheme is given by
\begin{align}\label{eq:scheme:sineGordon:sec3}
    \left\{
    \begin{alignedat}{1}
    &\delta_n^{\langle 2\rangle} U_k^{(n)}
    = \delta_k^{\langle 2 \rangle}
    \left(
        \frac{U_k^{(n+1)}+U_k^{(n-1)}}{2}
    \right) \\
    &\qquad \qquad 
    - \mathrm{sinc} \left( \frac{U_k^{(n+1)}-U_k^{(n-1)}}{2} \right)
    \sin \left( \frac{U_k^{(n+1)}+U_k^{(n-1)}}{2} \right),\\
    &\delta_n^{\langle 2\rangle} U_0^{(n)}
        - \delta_k^{\langle 1 \rangle}
        \left(
            \frac{U_0^{(n+1)}+U_0^{(n-1)}}{2}
        \right)
        = 0,\\
    &\delta_n^{\langle 2\rangle} U_K^{(n)}
       + \delta_k^{\langle 1 \rangle}
        \left(
            \frac{U_K^{(n+1)}+U_K^{(n-1)}}{2}
        \right)
        = 0.
    \end{alignedat}
    \right.
\end{align}

\subsection{General case}
We also give the energy conservation scheme
for more general nonlinear wave equations
with dynamic boundary conditions,
that is, the discrete version of the equations \eqref{eq:ndw}.
However, the existence of solutions and
error estimates are not discussed in this paper.
Let
$X_{\mathrm{d},k}: \mathbb{R}^2 \to \mathbb{R}$
and
$F: \mathbb{R}^2 \to \mathbb{R}$
be $C^1$ functions.
For
$\vect{U}^{(n)} = \{ U_k^{(n)} \}_{k=-1}^{K+1} \in \mathbb{R}^{K+3} \ (n=0,1,\ldots,N)$,
we define
discrete energy density functions
$G_{\mathrm{d},k}^+ (\vect{U}^{(n+1)}, \vect{U}^{(n)})$
and
$G_{\mathrm{d},k}^- (\vect{U}^{(n+1)}, \vect{U}^{(n)})$
by the form
\begin{align}\label{eq:gene:Gdk:+}
    G_{\mathrm{d},k}^+ (\vect{U}^{(n+1)}, \vect{U}^{(n)})
    &= X_{\mathrm{d},k}(\delta_k^+ U_k^{(n+1)}, \delta_k^+ U_k^{(n)}) + F(U_k^{(n+1)}, U_k^{(n)}),\\
\label{eq:gene:Gdk:-}
     G_{\mathrm{d},k}^- (\vect{U}^{(n+1)}, \vect{U}^{(n)})
    &= s_k^- \left( X_{\mathrm{d},k}(\delta_k^+ U_k^{(n+1)}, \delta_k^+ U_k^{(n)}) \right) + F(U_k^{(n+1)}, U_k^{(n)}).
\end{align}
Moreover, we define a discrete energy
$J_{\mathrm{d}}(\vect{U}^{(n+1)}, \vect{U}^{(n)})$
by
\begin{align}
    J_{\mathrm{d}}(\vect{U}^{(n+1)}, \vect{U}^{(n)})
    &= \sum_{k=0}^K{}'' \frac{1}{2} (\delta_n^+ U_k^{(n)})^2 \Delta x \\
    &\quad + 
    \frac{1}{2} \sum_{k=0}^{K-1}
    G_{\mathrm{d},k}^+(\vect{U}^{(n+1)}, \vect{U}^{(n)}) \Delta x
    + \frac{1}{2} \sum_{k=1}^K
    G_{\mathrm{d},k}^- (\vect{U}^{(n+1)}, \vect{U}^{(n)}) \Delta x \\
    &\quad
    + \frac{1}{2} (\delta_n^+ U_0^{(n)})^2
    + \frac{1}{2} (\delta_n^+ U_K^{(n)})^2.
\end{align}
From \eqref{eq:gene:Gdk:+} and \eqref{eq:gene:Gdk:-},
$J_{\mathrm{d}}(\vect{U}^{(n+1)}, \vect{U}^{(n)})$
is also written as
\begin{align}
    J_{\mathrm{d}}(\vect{U}^{(n+1)}, \vect{U}^{(n)})
    &=
    \sum_{k=0}^K{}'' \frac{1}{2} (\delta_n^+ U_k^{(n)})^2 \Delta x \\
    &\quad +
    \sum_{k=0}^{K-1} X_{\mathrm{d},k} (\delta_k^+ U_k^{(n+1)}, \delta_k^+ U_k^{(n)}) \Delta x
    + \sum_{k=0}^K{}'' F(U_k^{(n+1)}, U_k^{(n)}) \Delta x \\
    &\quad
    + \frac{1}{2} (\delta_n^+ U_0^{(n)})^2
    + \frac{1}{2} (\delta_n^+ U_K^{(n)})^2.
\end{align}
Our difference scheme which
preserves the energy
$J_{\mathrm{d}}(\vect{U}^{(n+1)}, \vect{U}^{(n)})$
is the following.
\begin{align}\label{eq:scheme:general}
    \left\{
    \begin{alignedat}{2}
    &\delta_n^{\langle 2\rangle} U_k^{(n)}
    = \delta_k^- \left(
        \frac{d(X_{\mathrm{d},k},\loc)}{d(\delta_k^+U_k^{(n+1)}, \delta_k^+ U_k^{(n)} : \delta_k^+ U_k^{(n)}, \delta_k^+ U_k^{(n-1)})}
        \right) & \ \\
    &\qquad \qquad \ -
        \frac{d(F,\loc)}{d(U_k^{(n+1)},U_k^{(n)} : U_k^{(n)}, U_k^{(n-1)})},
        \qquad (k=0,\ldots,K, n=1,\ldots,N-1),\\
    &\delta_n^{\langle 2\rangle} U_0^{(n)}
        - \mu_k^- \left(
            \frac{d(X_{\mathrm{d},0},\loc)}{d(\delta_k^+U_0^{(n+1)}, \delta_k^+ U_0^{(n)} : \delta_k^+ U_0^{(n)}, \delta_k^+ U_0^{(n-1)})}
            \right)
            = 0,
        \qquad (n=1,\ldots,N-1),\\
    &\delta_n^{\langle 2\rangle} U_K^{(n)}
       + \mu_k^- \left(
            \frac{d(X_{\mathrm{d},K},\loc)}{d(\delta_k^+U_K^{(n+1)}, \delta_k^+ U_K^{(n)} : \delta_k^+ U_K^{(n)}, \delta_k^+ U_K^{(n-1)})}
            \right)
            = 0,
        \qquad (n=1,\ldots,N-1).
    \end{alignedat}
    \right.
\end{align}
The proposed scheme above preserves the energy:
\begin{theorem}\label{thm:energy:conserv}
If
$\vect{U}^{(n)} = \{ U_k^{(n)} \}_{k=-1}^{K+1} \in \mathbb{R}^{K+3} \ (n=0,1,\ldots,N)$
satisfies the scheme \eqref{eq:scheme:general},
then
$\delta_n^- J_{\mathrm{d}}(\vect{U}^{(n+1)}, \vect{U}^{(n)}) = 0$
for all
$n =1,\ldots,N-1$.
\end{theorem}
\begin{proof}
We compute
\begin{align}
    \delta_n^- J_{\mathrm{d}}(\vect{U}^{(n+1)}, \vect{U}^{(n)})
    &=
    \sum_{k=0}^K{}'' \frac{1}{2} \delta_n^- (\delta_n^+ U_k^{(n)})^2 \Delta x \\
    &\quad +
    \sum_{k=0}^{K-1} \delta_n^- X_{\mathrm{d},k} (\delta_k^+ U_k^{(n+1)}, \delta_k^+U_k^{(n)}) \Delta x
    + \sum_{k=0}^K{}'' \delta_n^- F(U_k^{(n+1)},U_k^{(n)}) \Delta x \\
    &\quad + 
    \frac{1}{2} \delta_n^- (\delta_n^+ U_0^{(n)})^2
    + \frac{1}{2} \delta_n^- (\delta_n^+ U_K^{(n)})^2.
\end{align}
By the summation by parts,
the second term of the right-hand side is further calculated as
\begin{align}
    &\sum_{k=0}^{K-1} \delta_n^- X_{\mathrm{d},k} (\delta_k^+ U_k^{(n+1)}, \delta_k^+U_k^{(n)}) \Delta x \\
    &=
    \sum_{k=0}^{K-1}
    \frac{d(X_{\mathrm{d},k}, \loc)}{d(\delta_k^+ U_k^{(n+1)}, \delta_k^+ U_k^{(n)} : \delta_k^+ U_k^{(n)}, \delta_k^+ U_k^{(n-1)})}
    \delta_k^+ (\delta_n^{\langle 1 \rangle} U_k^{(n)}) \Delta x \\
    &=
    - \sum_{k=0}^{K}{}''
    \left\{
        \delta_k^- \left(
            \frac{d(X_{\mathrm{d},k}, \loc)}{d(\delta_k^+ U_k^{(n+1)}, \delta_k^+ U_k^{(n)} : \delta_k^+ U_k^{(n)}, \delta_k^+ U_k^{(n-1)})}
            \right)
        \delta_n^{\langle 1 \rangle} U_k^{(n)}
    \right\} \Delta x \\
    &\quad
    + \left[
    \mu_k^- \left(
        \frac{d(X_{\mathrm{d},k}, \loc)}{d(\delta_k^+ U_k^{(n+1)}, \delta_k^+ U_k^{(n)} : \delta_k^+ U_k^{(n)}, \delta_k^+ U_k^{(n-1)})}
        \right)
        \delta_n^{\langle 1 \rangle} U_k^{(n)}
    \right]_0^K.
\end{align}
The other terms can be computed in the same way
as in the proof of Theorem \ref{thm:energy:conserv:semi}.
From them and the equations \eqref{eq:scheme:general},
we conclude
\begin{align}
    &\delta_n^- J_{\mathrm{d}} (\vect{U}^{(n+1)}, \vect{U}^{(n)}) \\
    &=
    \sum_{k=0}^K{}''
    \left\{ 
    \delta_n^{\langle 2 \rangle} U_k^{(n)}
    + \frac{d(F,\loc)}{d(U_k^{(n+1)}, U_k^{(n)} : U_k^{(n)}, U_k^{(n-1)})} 
    \right. \\
    &\quad \quad \left.
    - \delta_k^- \left(
        \frac{d(X_{\mathrm{d},k}, \loc)}{d(\delta_k^+ U_k^{(n+1)}, \delta_k^+ U_k^{(n)} : \delta_k^+ U_k^{(n)}, \delta_k^+ U_k^{(n-1)})}
         \right)
         \right\}
         \delta_n^{\langle 1 \rangle} U_k^{(n)} \Delta x \\
    &\quad
    + \left\{
    \delta_n^{\langle 2 \rangle} U_0^{(n)}
    - \mu_k^- \left(
         \frac{d(X_{\mathrm{d},0}, \loc)}{d(\delta_k^+ U_0^{(n+1)}, \delta_k^+ U_0^{(n)} : \delta_k^+ U_0^{(n)}, \delta_k^+ U_0^{(n-1)})}
    \right)
    \right\}
    \delta_n^{\langle 1 \rangle} U_0^{(n)} \\
    &\quad
    + \left\{
    \delta_n^{\langle 2 \rangle} U_K^{(n)}
    + \mu_k^- \left(
         \frac{d(X_{\mathrm{d},K}, \loc)}{d(\delta_k^+ U_K^{(n+1)}, \delta_k^+ U_K^{(n)} : \delta_k^+ U_K^{(n)}, \delta_k^+ U_K^{(n-1)})}
    \right)
    \right\}
    \delta_n^{\langle 1 \rangle} U_K^{(n)} \\
    &= 0,
\end{align}
which completes the proof.
\end{proof}

\subsection{Example}
We consider the quasilinear string vibration equation
\begin{align}\label{eq:general:example:string}
    \left\{ \begin{aligned}[2]
    &\frac{\partial^2 u}{\partial t^2}
    = \frac{\partial}{\partial x}
    \left( \frac{u_x}{\sqrt{1+ u_x^2}} \right),
    &&(t,x) \in (0,T)\times (0,L),\\
    &\frac{\partial^2 u}{\partial t^2}(t,0)
    - \frac{u_x(t,0)}{\sqrt{1+ u_x(t,0)^2}} = 0,
    &&t\in (0,T),\\
    &\frac{\partial^2 u}{\partial t^2}(t,L)
    + \frac{u_x(t,L)}{\sqrt{1+ u_x(t,L)^2}} = 0,
    &&t\in (0,T).
    \end{aligned}\right.
\end{align}
The corresponding energy is defined by
\begin{align}
    J(u(t))
    &=
    \int_0^L \left(
    \frac{1}{2} u_t(t,x)^2
    + \sqrt{1+u_x(t,x)^2}
    \right)\,dx
    + \frac{1}{2}u_t(t,0)^2 + \frac{1}{2}u_t(t,L)^2.
\end{align}
We define a function
$X_{\mathrm{d},k}(a,b): \mathbb{R}^2 \to \mathbb{R}$
by
\begin{align}
    X_{\mathrm{d},k}(a,b)
    := \frac{1}{2} \left( \sqrt{1+a^2} + \sqrt{1+b^2} \right),
\end{align}
and the discrete energy by
\begin{align}
    J_{\mathrm{d}} (\vect{U}^{(n+1)}, \vect{U}^{(n)})
    &=
    \sum_{k=0}^K{}'' \frac{1}{2} (\delta_n^+ U_k^{(n)})^2 \Delta x \\
    &\quad +
    \sum_{k=0}^{K-1} 
    \frac{1}{2}
    \left(
        \sqrt{1 + (\delta_k^+ U_k^{(n+1)})^2}
        + \sqrt{1+ (\delta_k^+ U_k^{(n)})^2}
    \right)
    \Delta x \\
    &\quad
    + \frac{1}{2} (\delta_n^+ U_0^{(n)})^2
    + \frac{1}{2} (\delta_n^+ U_K^{(n)})^2.
\end{align}
Then, we compute
\begin{align}
    &\frac{d(X_{\mathrm{d},k}, \loc)}{d(\delta_k^+ U_k^{(n+1)}, \delta_k^+ U_k^{(n)} : \delta_k^+ U_k^{(n)}, \delta_k^+ U_k^{(n-1)})}
    =
    \frac{\delta_k^+ U_k^{(n+1)}+\delta_k^+ U_k^{(n-1)}}{
    \sqrt{1+(\delta_k^+ U_k^{(n+1)})^2}
    + \sqrt{1+(\delta_k^+ U_k^{(n-1)})^2}
    },
\end{align}
and hence, the energy conservation scheme is given by
\begin{align}\label{eq:gene:scheme:string}
    \left\{
    \begin{alignedat}{1}
    &\delta_n^{\langle 2\rangle} U_k^{(n)}
    =
    \delta_k^- \left(
    \frac{\delta_k^+ U_k^{(n+1)}+\delta_k^+ U_k^{(n-1)}}{
    \sqrt{1+(\delta_k^+ U_k^{(n+1)})^2}
    + \sqrt{1+(\delta_k^+ U_k^{(n-1)})^2}
    }
    \right)
    ,\\
    &\delta_n^{\langle 2\rangle} U_0^{(n)}
        - \mu_k^-
        \left(
            \frac{\delta_k^+ U_0^{(n+1)}+\delta_k^+ U_0^{(n-1)}}{
    \sqrt{1+(\delta_k^+ U_0^{(n+1)})^2}
    + \sqrt{1+(\delta_k^+ U_0^{(n-1)})^2}
    }
        \right)
        = 0,\\
    &\delta_n^{\langle 2\rangle} U_K^{(n)}
       + \mu_k^-
        \left(
            \frac{\delta_k^+ U_K^{(n+1)}+\delta_k^+ U_K^{(n-1)}}{
    \sqrt{1+(\delta_k^+ U_K^{(n+1)})^2}
    + \sqrt{1+(\delta_k^+ U_K^{(n-1)})^2}
    }
        \right)
        = 0.
    \end{alignedat}
    \right.
\end{align}

\section{Existence of solutions}
In this section, we prove the existence of
the solution for the difference scheme derived in
Section 3
in the case where the equation is semilinear, 
that is,
\begin{align}\label{eq:sdw:sec5}
    \left\{ \begin{aligned}[2]
    &\frac{\partial^2 u}{\partial t^2} = \frac{\partial^2 u}{\partial x^2} - \tilde{F}'(u),
    &&(t,x) \in (0,T)\times (0,L),\\
    &\frac{\partial^2 u}{\partial t^2}(t,0)
    - \frac{\partial u}{\partial x}(t,0) = 0,
    &&t\in (0,T),\\
    &\frac{\partial^2 u}{\partial t^2}(t,L)
    + \frac{\partial u}{\partial x}(t,L) = 0,
    &&t\in (0,T),
    \end{aligned}\right.
\end{align}
where $\tilde{F}$ is sufficiently smooth. 
For simplicity, we assume that the function 
$F: \mathbb{R}^2 \to \mathbb{R}$
has the form
\begin{align}\label{eq:F:sec5}
    F(a,b) := \frac{1}{2} \left( \tilde{F}(a) + \tilde{F}(b) \right).
\end{align}
Under the assumption, the four-point difference quotient is rewritten into 
the two-point difference quotient as follows: 
\begin{align*}
    \frac{d(F,\loc)}{d(U_k^{(n+1)},U_k^{(n)} : U_k^{(n)}, U_k^{(n-1)})} &= \frac{d \tilde{F}}{d(U_k^{(n+1)}, U_k^{(n-1)})} \\  
    &=\begin{dcases} 
    \frac{\tilde{F}(U_k^{(n+1)}) - \tilde{F}(U_k^{(n-1)})}{U_k^{(n+1)} -  U_k^{(n-1)}} &(U_k^{(n+1)} \neq U_k^{(n-1)}), \\
    \tilde{F}'(U_k^{(n-1)}) &(U_k^{(n+1)} =  U_k^{(n-1)}).
    \end{dcases}
\end{align*}
Hence the scheme \eqref{eq:scheme:semi} is rewritten as 
\begin{align}\label{eq:scheme:semi:sec4}
    \left\{
    \begin{alignedat}{1}
    &\delta_n^{\langle 2\rangle} U_k^{(n)}
    = \delta_k^{\langle 2 \rangle}
    \left(
        \frac{U_k^{(n+1)}+U_k^{(n-1)}}{2}
    \right)
    - \frac{d \tilde{F}}{d(U_k^{(n+1)}, U_k^{(n-1)})},\\
    &\delta_n^{\langle 2\rangle} U_0^{(n)}
        - \delta_k^{\langle 1 \rangle}
        \left(
            \frac{U_0^{(n+1)}+U_0^{(n-1)}}{2}
        \right)
        = 0,\\
    &\delta_n^{\langle 2\rangle} U_K^{(n)}
       + \delta_k^{\langle 1 \rangle}
        \left(
            \frac{U_K^{(n+1)}+U_K^{(n-1)}}{2}
        \right)
        = 0,
    \end{alignedat}
    \right.
\end{align}
where
$k=0,\ldots,K$
and
$n=1,\ldots,N-1$.
Although, a mathematical treatment for the structure-preserving scheme of quasilinear equations is studied in \cite{yo-ka}, 
it is restricted the problem to semi-discrete case under simple boundary conditions. 
As is shown in also e.g. \cite{ko-lu}, it still remains some difficult to treat quasilinear equations even if under 
basic boundary conditions. 
Therefore, in what follows,
we focus our analysis on the
semilinear problem \eqref{eq:scheme:semi:sec4}.

Denote
$\vect{U} = \{ U_k \}_{k=0}^{K},
\widetilde{\vect{U}} = \{ \widetilde{U}_k \}_{k=0}^{K}
\in \mathbb{R}^{K+1}$.
For given vectors
$\vect{U}^{(n)} = \{ U_k^{(n)} \}_{k=-1}^{K+1},
\vect{U}^{(n-1)} = \{ U_k^{(n-1)} \}_{k=-1}^{K+1}
\in \mathbb{R}^{K+3}$,
let us define the nonlinear mapping $\Phi:\vect{U} \mapsto \widetilde{\vect{U}}$ by 
\begin{align}\label{eq:nmp}
    \left\{
    \begin{alignedat}{1}
    &\frac{\widetilde{U}_k - 2 U_k^{(n)} +  U_k^{(n-1)}}{(\Delta t)^2}  
    = \delta_k^{\langle 2 \rangle}
    \left(
        \frac{\widetilde{U}_k+U_k^{(n-1)}}{2}
    \right)
    - \frac{\widetilde{U}_k+U_k^{(n-1)}}{2} \\
    &\ \qquad \qquad \qquad \qquad \qquad + \frac{U_k+U_k^{(n-1)}}{2}
    - \frac{d \tilde{F}}{d(U_k, U_k^{(n-1)})},\\
    &\frac{\widetilde{U}_0 - 2 U_0^{(n)} +  U_0^{(n-1)}}{(\Delta t)^2} 
        - \delta_k^{\langle 1 \rangle}
        \left(
            \frac{\widetilde{U}_0+U_0^{(n-1)}}{2}
        \right)
        = 0,\\
    &\frac{\widetilde{U}_K - 2 U_K^{(n)} +  U_K^{(n-1)}}{(\Delta t)^2} 
       + \delta_k^{\langle 1 \rangle}
        \left(
            \frac{\widetilde{U}_K+U_K^{(n-1)}}{2}
        \right)
        = 0,
    \end{alignedat}
    \right.
\end{align}
where
$k=0,\ldots,K$.
We remark that,
the above scheme includes
the artificial terms
$\widetilde{U}_{-1}$
and
$\widetilde{U}_{K+1}$,
while they are determined from
$\{ \widetilde{U}_k \}_{k=0}^{K}$
as in the following argument.

We shall prove the mapping is contractive in 
\[
X_M := \{ \vect{U} \in \mathbb{R}^{K+1} \mid \| (\vect{U}, \vect{U} - \vect{U}^{(n)}) \|_X^2 \leq 3 M_n^2 \}, 
\]
where
$M_n := \| (\vect{U}^{(n-1)}, \vect{U}^{(n)}-\vect{U}^{(n-1)}) \|_X$
and 
\[
\| (\vect{U}, \vect{V}) \|_X := \sqrt{\| \vect{U} \|_{H^1}^2 + \frac{1}{(\Delta t)^2} \| \vect{V} \|_{2}^2
+ \frac{1}{(\Delta t)^2}( |V_0|^2 + |V_K|^2 )
}. 
\]

\begin{theorem}[Local existence]\label{thm:exist:sol}
Assume that 
$\tilde{F} \in C^2$.
Let
$n \in \mathbb{N}$.
For any given
$\{(U_k^{(n)}, U_k^{(n-1)}) \}_{k=-1}^{K+1}$,
there exists a constant
$R_1 = R_1(M_n) > 0$
such that for
$\Delta t < R_1$
there exists a unique solution 
$\{U_k^{(n+1)}\}_{k=-1}^{K+1}$
for \eqref{eq:scheme:semi:sec4}
in $X_M$. 
\end{theorem}
\begin{proof}
Let us first show the mapping is well-defined.   
We rewrite the above difference scheme in
the vector form. 
We first obtain from
$\eqref{eq:nmp}_2$ and
$\eqref{eq:nmp}_3$ that 
\begin{gather*}
\frac{\widetilde{U}_{-1}+U_{-1}^{(n-1)}}{2}
= \frac{\widetilde{U}_1+U_1^{(n-1)}}{2}
- 2 \Delta x \frac{\widetilde{U}_0 - 2 U_0^{(n)} +  U_0^{(n-1)}}{(\Delta t)^2}, \\
\frac{\widetilde{U}_{K+1}+U_{K+1}^{(n-1)}}{2}
= \frac{\widetilde{U}_{K-1}+U_{K-1}^{(n-1)}}{2}
- 2 \Delta x \frac{\widetilde{U}_K - 2 U_K^{(n)} +  U_K^{(n-1)}}{(\Delta t)^2}. 
\end{gather*}
Substituting these into $0$-th and $K$-th equations of $\eqref{eq:nmp}_1$ yields 
\begin{gather*}
\frac{\widetilde{U}_0 - 2 U_0^{(n)} +  U_0^{(n-1)}}{(\Delta t)^2}  
    = \frac{2}{(\Delta x)^2}
    \left(
        \frac{\widetilde{U}_1+U_1^{(n-1)}}{2}
        -  \frac{\widetilde{U}_0+U_0^{(n-1)}}{2} \right)
    - \frac{\widetilde{U}_0+U_0^{(n-1)}}{2} \\
    \qquad \qquad - \frac{2}{\Delta x} \cdot \frac{\widetilde{U}_0 - 2 U_0^{(n)} +  U_0^{(n-1)}}{(\Delta t)^2}
    + \frac{U_0+U_0^{(n-1)}}{2}
    -\frac{d \tilde{F}}{d(U_0, U_0^{(n-1)})},\\
\frac{\widetilde{U}_K - 2 U_K^{(n)} +  U_K^{(n-1)}}{(\Delta t)^2}  
    = - \frac{2}{(\Delta x)^2}
    \left(
        \frac{\widetilde{U}_K+U_K^{(n-1)}}{2}
        -  \frac{\widetilde{U}_{K-1}+U_{K-1}^{(n-1)}}{2} \right)
    - \frac{\widetilde{U}_K+U_K^{(n-1)}}{2} \\
    \qquad \qquad
    - \frac{2}{\Delta x} \cdot \frac{\widetilde{U}_K - 2 U_K^{(n)} +  U_K^{(n)}}{(\Delta t)^2}
    + \frac{U_K+U_K^{(n-1)}}{2}
    -\frac{d \tilde{F}}{d(U_K, U_K^{(n-1)})},
\end{gather*}
that is,
\begin{align*}
    &\widetilde{U}_0 - \frac{(\Delta t)^2}{(\Delta x)^2} (\widetilde{U}_1 - \widetilde{U}_0)
    + \frac{2}{\Delta x} \widetilde{U}_0 + \frac{(\Delta t)^2}{2} \widetilde{U}_0 \\
    &=
    2 U_0^{(n)} -  U_0^{(n-1)} + \frac{(\Delta t)^2}{(\Delta x)^2} (U_1^{(n-1)} - U_0^{(n-1)}) \\
    &\quad
    + \frac{2}{\Delta x} (2 U_0^{(n)} -  U_0^{(n-1)}) + \frac{(\Delta t)^2}{2} U_0 
    - (\Delta t)^2 \cdot \frac{d \tilde{F}}{d(U_0, U_0^{(n-1)})},\\
    &\widetilde{U}_K + \frac{(\Delta t)^2}{(\Delta x)^2} (\widetilde{U}_K - \widetilde{U}_{K-1})
    + \frac{2}{\Delta x} \widetilde{U}_K + \frac{(\Delta t)^2}{2} \widetilde{U}_K \\
    &=
    2 U_K^{(n)} -  U_K^{(n-1)} + \frac{(\Delta t)^2}{(\Delta x)^2} (U_{K-1}^{(n-1)} - U_K^{(n-1)}) \\
    &\quad
    - \frac{2}{\Delta x} (2 U_K^{(n)} -  U_K^{(n-1)}) + \frac{(\Delta t)^2}{2} U_K
    - (\Delta t)^2 \cdot \frac{d \tilde{F}}{d(U_K, U_K^{(n-1)})}.
\end{align*}
In a similar manner, we rewrite the $k$-th equation of \eqref{eq:nmp} for $k=1,2,\ldots, K-1$ as follows: 
\begin{align*}
&\widetilde{U}_k - \frac{(\Delta t)^2}{2  (\Delta x)^2} (\widetilde{U}_{k+1} -  2\widetilde{U}_k + \widetilde{U}_{k-1}) + \frac{(\Delta t)^2}{2} \widetilde{U}_k \\
&= 2 U_k^{(n)} -  U_k^{(n-1)} + \frac{(\Delta t)^2}{2  (\Delta x)^2} (U_{k+1}^{(n-1)} -  2U_k^{(n-1)} + U_{k-1}^{(n-1)}) + \frac{(\Delta t)^2}{2} U_k
    - (\Delta t)^2 \cdot \frac{d \tilde{F}}{d(U_k, U_k^{(n-1)})}.
\end{align*}
We thus rewrite the equations \eqref{eq:nmp}
to the vector form: 
\begin{align}\label{eq:scheme:vec}
    \left\{ \left(1 + \frac{(\Delta t)^2}{2} \right) E - A\right\} \widetilde{\vect{U}}
    &= \vect{G}, 
\end{align}
where $E$ is
the $(K+1)$-dimensional identity matrix and the $(K+1)$-dimensional matrix $A$ is defined by 
\begin{align}
    A &:=
    \begin{pmatrix}
    -2 \alpha - \beta & 2 \alpha & 0 &0 &\cdots & 0\\
    \alpha & - 2 \alpha & \alpha & 0 &\cdots & 0\\
    0 &\alpha &-2 \alpha &\alpha & \cdots & 0\\
    \vdots &\ &\ddots&\ddots &\ddots &\vdots \\
    0&\cdots &0 &\alpha &-2 \alpha &\alpha \\
    0&\cdots&0&0 &2 \alpha &-2 \alpha -\beta
    \end{pmatrix},
\end{align}
for $\alpha := \frac{(\Delta t)^2}{2(\Delta x)^2}$, $\beta := \frac{2}{\Delta x}$, 
and $\vect{G}$ means the right-hand side vector determined by given values $\vect{U}^{(n)}, \vect{U}^{(n-1)}, \vect{U}$. 
From the completely same argument as in \cite{5}, we can show the matrix $A$ is 
negative definite
with respect to the inner product with the trapezoidal rule. 
We thus obtain the matrix $\left\{ \left(1 + \frac{(\Delta t)^2}{2} \right) E - A \right\}$ is 
positive definite, and hence is non-singular. 
Therefore, the mapping $\Phi$ is well-defined. 

Next, we show the mapping $\Phi$ is contractive on $X$. 
Multiplying $\eqref{eq:nmp}_1$ by
$\frac{\widetilde{U}_k - U_k^{(n-1)}}{\Delta t}$ and 
summing the resulting equations following the trapezoidal rule, 
we obtain
\begin{align*}
&\frac{1}{\Delta t} \left( \left\| \frac{\widetilde{\vect{U}} -\vect{U}^{(n)}}{\Delta t} \right\|_2^2 - \left\| \frac{\vect{U}^{(n)} -\vect{U}^{(n-1)}}{\Delta t} \right\|_2^2 \right)
+ \frac{1}{2\Delta t} \left( \left\| \widetilde{\vect{U}} \right\|_2^2 - \left\| \vect{U}^{(n-1)} \right\|_2^2 \right) \\
&= \sum_{k=0}^K{}'' \frac{\widetilde{U}_k - U_k^{(n-1)}}{\Delta t} 
\cdot \left( \delta_k^{\langle 2 \rangle}
    \left(
        \frac{\widetilde{U}_k+U_k^{(n-1)}}{2}
    \right) + \frac{U_k+U_k^{(n-1)}}{2}
    -   \frac{d \tilde{F}}{d(U_k, U_k^{(n-1)})} \right) \Delta x \\
&\leq - \frac{1}{2\Delta t} \left( \| D \widetilde{\vect{U}} \|^2 - \| D \vect{U}^{(n-1)} \|^2 \right)   
+ \left[ \frac{\widetilde{U}_k - U_k^{(n-1)}}{\Delta t} \delta_k^{\langle 1 \rangle} \left( \frac{\widetilde{U}_k+U_k^{(n-1)}}{2} \right) \right]_0^K \\
&+ \left\| \frac{\widetilde{\vect{U}} -\vect{U}^{(n-1)}}{\Delta t} \right\|_2 
\cdot \left( \left\| \frac{\vect{U} + \vect{U}^{(n-1)}}{2} \right\|_2 + \left\|  \frac{d \tilde{F}}{d(\vect{U}, \vect{U}^{(n-1)})} \right\|_2 \right). 
\end{align*}
Then it follows from the equations $\eqref{eq:nmp}_2$ and $\eqref{eq:nmp}_3$ that 
\begin{align*}
&\frac{1}{\Delta t} \left( \left\| \frac{\widetilde{\vect{U}} -\vect{U}^{(n)}}{\Delta t} \right\|_2^2 - \left\| \frac{\vect{U}^{(n)} -\vect{U}^{(n-1)}}{\Delta t} \right\|_2^2 \right)
+ \frac{1}{2\Delta t} \left( \left\| \widetilde{\vect{U}} \right\|_{H^1}^2 - \left\| \vect{U}^{(n-1)} \right\|_{H^1}^2 \right) \\
&\quad + \frac{1}{\Delta t}\left( \left| \frac{\widetilde{U}_0 - U_0^{(n)}}{\Delta t} \right|^2 - \left| \frac{U_0^{(n)} - U_0^{(n-1)}}{\Delta t} \right|^2 \right)
+ \frac{1}{\Delta t}\left( \left| \frac{\widetilde{U}_K - U_K^{(n)}}{\Delta t} \right|^2 - \left| \frac{U_K^{(n)} - U_K^{(n-1)}}{\Delta t} \right|^2 \right) \\
&\leq \left( \left\| \frac{\widetilde{\vect{U}} -\vect{U}^{(n)}}{\Delta t} \right\|_2 + \left\| \frac{\vect{U}^{(n)} -\vect{U}^{(n-1)}}{\Delta t} \right\|_2 \right) \cdot \left( \left\| \frac{\vect{U} + \vect{U}^{(n-1)}}{2} \right\|_2 + \left\|   \frac{d \tilde{F}}{d(\vect{U}, \vect{U}^{(n-1)})} \right\|_2 \right) \\
&\leq \frac{1}{2\Delta t} \left\| \frac{\widetilde{\vect{U}} -\vect{U}^{(n)}}{\Delta t} \right\|_2^2 
    + \frac{1}{4\Delta t} \left\| \frac{\vect{U}^{(n)} -\vect{U}^{(n-1)}}{\Delta t} \right\|_2^2 \\
&\qquad \qquad    + \frac{3 \Delta t}{2}\left( \left\| \frac{\vect{U} + \vect{U}^{(n-1)}}{2} \right\|_2 + \left\|   \frac{d \tilde{F}}{d(\vect{U}, \vect{U}^{(n-1)})} \right\|_2 \right)^2 \\
&\leq \frac{1}{2\Delta t} \left\| \frac{\widetilde{\vect{U}} -\vect{U}^{(n)}}{\Delta t} \right\|_2^2 
    + \frac{1}{4\Delta t} \left\| \frac{\vect{U}^{(n)} -\vect{U}^{(n-1)}}{\Delta t} \right\|_2^2 \\
&\qquad \qquad    + \frac{3 \Delta t}{2}\left( \left\| \vect{U}  \right\|_2^2 + \left\|  \vect{U}^{(n-1)} \right\|_2^2 + 2 \left\|  \frac{d \tilde{F}}{d(\vect{U}, \vect{U}^{(n-1)})} \right\|_2^2 \right).
\end{align*}
From Lemma \ref{prop-f} (1) we see that
\begin{align*}
\left\|  \frac{d \tilde{F}}{d(\vect{U}, \vect{U}^{(n-1)})} \right\|_2 \leq \min \left\{  C_{\tilde{F}'}(\sqrt{3} C_S M_n) \sqrt{L}, \  |F'(0)| \sqrt{L} + (1+\sqrt{3}) C_{\tilde{F}''}(\sqrt{3} C_S M_n) M_n \right\}. 
\end{align*}
Setting
\begin{align*}
C_{into}(M_n) := \sqrt{4 M_n^2 + 
2 \left(\min \left\{  C_{\tilde{F}'}(\sqrt{3} C_S M_n) \sqrt{L}, \  |F'(0)| \sqrt{L} + (1+\sqrt{3}) C_{\tilde{F}''}(\sqrt{3} C_S M_n) M_n \right\} \right)^2 },
\end{align*}
we arrive at  
\begin{align*}
&\left\| \frac{\widetilde{\vect{U}} -\vect{U}^{(n)}}{\Delta t} \right\|_2^2
+ \left\| \widetilde{\vect{U}} \right\|_{H^1}^2
 + \left( \left| \frac{\widetilde{U}_0 - U_0^{(n)}}{\Delta t} \right|^2 + \left| \frac{\widetilde{U}_K - U_K^{(n)}}{\Delta t} \right|^2  \right)  \\
&\leq \frac{5}{2} \left\| \frac{\vect{U}^{(n)} -\vect{U}^{(n-1)}}{\Delta t} \right\|_2^2 + \left\| \vect{U}^{(n-1)} \right\|_{H^1}^2 
+ \left(
\left| \frac{U_0^{(n)} - U_0^{(n-1)}}{\Delta t} \right|^2
+ \left| \frac{U_K^{(n)} - U_K^{(n-1)}}{\Delta t} \right|^2
\right)  \\
&\quad + 3 (\Delta t)^2 C_{into}(M_n)^2. 
\end{align*}
Choosing $\Delta t$ small satisfying 
\begin{equation}\label{eq:into}
   3 (\Delta t)^2 C_{into}(M_n)^2 \leq \frac{1}{2} M_n^2 
\end{equation} 
implies that $\Phi$ maps into $X_M$ from $X_M$. 

Next, let us show the mapping $\Phi$ is contractive.
For arbitrary
$\vect{U}_1 = \{ U_{1,k} \}_{k=0}^K \in X_M$
and
$\vect{U}_2 = \{ U_{2,k} \}_{k=0}^K \in X_M$,
we write
$\widetilde{\vect{U}}_1 =
\{ \widetilde{U}_{1,k} \}_{k=0}^K := \Phi[\vect{U}_1]$ 
and $\widetilde{\vect{U}}_2 = \{ \widetilde{U}_{2,k} \}_{k=0}^K := \Phi[\vect{U}_2]$, respectively.  
Subtracting these relations yields 
\begin{align}\label{eq:nmp2}
    \left\{
    \begin{alignedat}{1}
    &\frac{\widetilde{U}_{1,k} - \widetilde{U}_{2,k}}{(\Delta t)^2}  
    = \delta_k^{\langle 2 \rangle}
    \left(
        \frac{\widetilde{U}_{1,k} - \widetilde{U}_{2,k}}{2}
    \right)
    - \frac{\widetilde{U}_{1,k} - \widetilde{U}_{2,k}}{2} \\
    &\qquad \qquad \qquad + \frac{U_{1,k} - U_{2,k}}{2}
    - \left\{ \frac{d \tilde{F}}{d(U_{1,k}, U_k^{(n-1)})}
    - \frac{d \tilde{F}}{d(U_{2,k}, U_k^{(n-1)})} \right\},\\
    &\frac{\widetilde{U}_{1,0} - \widetilde{U}_{2,0}}{(\Delta t)^2} 
        - \delta_k^{\langle 1 \rangle}
        \left(
            \frac{\widetilde{U}_{1,0} - \widetilde{U}_{2,0}}{2}
        \right)
        = 0,\\
    &\frac{\widetilde{U}_{1,K} - \widetilde{U}_{2,K}}{(\Delta t)^2} 
       + \delta_k^{\langle 1 \rangle}
        \left(
            \frac{\widetilde{U}_{1,K} - \widetilde{U}_{2,K}}{2}
        \right)
        = 0.
    \end{alignedat}
    \right.
\end{align}
In a similar manner as above, multiplying
$\eqref{eq:nmp2}_1$
by $2(U_{1,k}-U_{2,k})$ we obtain 
\begin{align*}
    &2 \left\| \frac{\widetilde{\vect{U}_1} - \widetilde{\vect{U}_2}}{\Delta t} \right\|_2^2 
    + \left\| \widetilde{\vect{U}_1} - \widetilde{\vect{U}_2} \right\|_{H^1}^2 \\
    &= \left[ (\widetilde{U}_{1,k}-\widetilde{U}_{2,k}) \delta_k^{\langle 1 \rangle} (\widetilde{U}_{1,k}-\widetilde{U}_{2,k}) \right]_0^K 
    + \sum_{k=0}^K{}'' (\widetilde{U}_{1,k}-\widetilde{U}_{2,k}) \cdot (U_{1,k}-U_{2,k}) \Delta x \\
    &\qquad + 2\sum_{k=0}^K{}'' (\widetilde{U}_{1,k}-\widetilde{U}_{2,k}) \cdot \left\{ \frac{d \tilde{F}}{d(U_{1,k}, U_k^{(n-1)})}
    - \frac{d \tilde{F}}{d(U_{2,k}, U_k^{(n-1)})} \right\} \Delta x.
\end{align*}
By virtue of
$\eqref{eq:nmp2}_2$ and
$\eqref{eq:nmp2}_3$,
we have
\begin{align*}
    &\left\| \frac{\widetilde{\vect{U}_1} - \widetilde{\vect{U}_2}}{\Delta t} \right\|_2^2 
    + \left\| \widetilde{\vect{U}_1} - \widetilde{\vect{U}_2} \right\|_{H^1}^2
    +
    \left| \frac{\widetilde{U}_{1,0} - \widetilde{U}_{2,0}}{\Delta t} \right|^2
    +
    \left| \frac{\widetilde{U}_{1,K} - \widetilde{U}_{2,K}}{\Delta t} \right|^2
    \\
    &\leq \frac{(\Delta t)^2}{2}\left\| \vect{U}_1 - \vect{U}_2 \right\|_2^2 
    + 2
    (\Delta t)^2 \left\|   \frac{d \tilde{F}}{d(\vect{U}_1, \vect{U}^{(n-1)})} 
    -   \frac{d \tilde{F}}{d(\vect{U}_2, \vect{U}^{(n-1)})}
    \right\|_2^2.
\end{align*}
From Lemma \ref{prop-f} (2) we see that 
\[
\left\|   \frac{d \tilde{F}}{d(\vect{U}_1, \vect{U}^{(n-1)})} 
    -   \frac{d \tilde{F}}{d(\vect{U}_2, \vect{U}^{(n-1)})}
    \right\|_2
    \leq \frac{1}{2} C_{\tilde{F}''}(\sqrt{3} C_S M_n) \| \vect{U}_1 - \vect{U}_2 \|_2. 
\]
Setting 
\[
C_{contr}(M_n) := \sqrt{\frac{1}{2} + 
\frac{1}{2} C_{\tilde{F}''}(\sqrt{3} C_S M_n)^2}, 
\]
we see that 
\begin{align*}
    \left\| \frac{\widetilde{\vect{U}_1} - \widetilde{\vect{U}_2}}{\Delta t} \right\|_2^2 
    + \left\| \widetilde{\vect{U}_1} - \widetilde{\vect{U}_2} \right\|_{H^1}^2 
    \leq (\Delta t)^2 C_{contr}(M_n)^2 \left\| \vect{U}_1 - \vect{U}_2 \right\|_2^2. 
\end{align*}
By taking $\Delta t$ small satisfying 
\begin{equation}\label{eq:contr}
   \Delta t C_{contr}(M_n) < 1,  
\end{equation} 
the mapping $\Phi$ is contractive in $X_M$.

Now, in view of \eqref{eq:into} and \eqref{eq:contr},
we set
$R_1(M_n):=\min \left\{ \frac{1}{C_{contr}(M_n)}, \frac{M_n}{\sqrt{6} C_{into}(M_n)} \right\}$.
Under the smallness assumption of
$\Delta t < R_1$,
the unique existence of fixed point of
$\Phi$
is assured. 
Consequently,
the fixed point is $\vect{U}^{(n+1)}$ which we seek.
\end{proof}

Observe that,
for example,
in the case of
$F(u)=u^4/4$
the assumption for
$\Delta t$ requires
$\Delta t \leq C/(1+ M_n^2)$
to apply Theorem \ref{thm:exist:sol}.
This indicates that, if we have
an a priori estimate
independent of $n$ for the quantity
$M_n$,
we can construct a global solution
by applying Theorem \ref{thm:exist:sol}
repeatedly.
In fact, combining the theorem with the energy estimate, we obtain the global existence of solution. 
\begin{theorem}[Global existence]\label{thm:exist:gl-sol}
Assume that $\tilde{F} \in C^2$ is bounded from below, i.e., $\tilde{F} \geq -B_F$ for some constant $B_F \in \mathbb{R}$. 
Then for any $\Delta t < R$ for some $R$ determined by $(\vect{U}^{(0)}, \vect{U}^{(1)})$, $T$ and $L$ 
there exists a unique solution 
$\{U_k^{(n)}\}_{k=-1}^{K+1}$ satisfying \eqref{eq:scheme:semi:sec4} for
$n = 2,\ldots, N$. 
\end{theorem}
\begin{proof}
Once we show that $M_n$ is independent of $n$, namely,
for some constant $M$ independent of $n$ the estimate
\begin{align}\label{eq:M:sec4}
    \| (\vect{U}^{(n-1)}, \vect{U}^{(n)}-\vect{U}^{(n-1)}) \|_X
    \leq M
\end{align}
holds,
the claim is proved by exchanging the assumption $\Delta t < R_1(M_n,L)$ in Theorem \ref{thm:exist:sol} to $\Delta t < R_1(M,L)=:R$. 
It is easily seen from the energy conservation law that 
\begin{align}
 &\left\| \delta_n^+ \vect{U}^{(n)} \right\|_2^2 
+ \sum_{k=0}^{K-1} 
    \left(
        \frac{(\delta_k^+ U_k^{(n+1)})^2 + (\delta_k^+ U_k^{(n)})^2}{2}
    \right)
    \Delta x
+ \left| U_0^{(n)} \right|^2 + \left| U_K^{(n)} \right|^2 \\
&\leq 2 J_d(\vect{U}^{(1)}, \vect{U}^{(0)}) + 2B_F L =:E_0. 
\end{align}
From a direct calculation,
\begin{align*}
    \left\| \vect{U}^{(n)} \right\|_2 
    &\leq  \left\| \vect{U}^{(n-1)} \right\|_2 + \Delta t \left\| \delta_n^+ \vect{U}^{(n-1)} \right\|_2 \\
    &\leq  \left\| \vect{U}^{(n-1)} \right\|_2 + \Delta t \sqrt{E_0} \\
    &\leq \cdots \\
    &\leq \left\| \vect{U}^{(0)} \right\|_2 + T \sqrt{E_0}. 
\end{align*}
Then we obtain 
\[
    \| (\vect{U}^{(n-1)}, \vect{U}^{(n)}-\vect{U}^{(n-1)}) \|_X^2
    \leq
    E_0 + \left( \left\| \vect{U}^{(0)} \right\|_2 + T \sqrt{E_0} \right)^2, 
\]
which implies the claim
\eqref{eq:M:sec4}
by taking
$M^2 := 2 E_0 + \left( \left\| \vect{U}^{(0)} \right\|_2 + T \sqrt{E_0} \right)^2$. 
\end{proof}

Let us consider the sequence $\{ \vect{U}_j \}$ defined by 
$\vect{U}_{j+1} =\Phi[\vect{U}_j]$, $\vect{U}_0 = \vect{U}^{(n)}$. 
Then under the assumption of Theorem \ref{thm:exist:gl-sol} we see that $\vect{U}_j \to \vect{U}^{(n)}$ as $j \to \infty$. 
We thus utilize the proof in the actual numerical simulation.

\begin{remark}\label{rem:exist:sol}
Suppose that the nonlinearity $\tilde{F}$
satisfies for some $\varepsilon>0$ 
\begin{equation}\label{ass:nl}
\sum_{k=0}^K{}'' (\tilde{F}(u_k) - \varepsilon u_k^2) \Delta x \geq  0
\qquad \text{for any} \quad \{ u_k \}_{k=0}^k \in \mathbb{R}^{K+1}
\end{equation}
in addition to the assumptions of Theorem \ref{thm:exist:gl-sol}. 
Then 
the solution can be extended infinitely $n=0,1,2,\ldots$. 
Indeed, in this case from only the energy conservation law we see that 
\[
 \left\| \delta_n^+ \vect{U}^{(n)} \right\|_2^2 
+ \min\{ 1, \varepsilon \} \left\| \delta_n^+ \vect{U}^{(n)} \right\|_{H^1}^2
\leq 2 J_d(\vect{U}^{(1)}, \vect{U}^{(0)}) + 2B_F L =:M, 
\]
which implies that $M$ no longer depends on $T$. 
For example, $\tilde{F}(u) = u^4/4 + u^2/2$ (Klein-Gordon type equation) satisfies \eqref{ass:nl} unconditionally for any initial value. 
\end{remark}

\section{Error estimates}
In this section, we discuss
estimates for the error between
the approximate solution and the exact solution. 
We impose the regularity assumption for the nonlinearity
$\widetilde{F} \in C^3$,
which is natural because we will assume
$u \in C^4$
for the exact solution. 
Let $u_k^{(n)}$ denote 
$u(n\Delta t, k \Delta x)$,
that is, the values of $u$ on the lattice,
and define the error by
$e_k^{(n)} = U_k^{(n)} - u_k^{(n)}$.
The vector forms are denoted by
$\vect{u}^{(n)}$ and $\vect{e}^{(n)}$.

\begin{theorem}\label{thm:error}
Let $\tilde{F} \in C^3$ and $u \in C^4([0,T]\times [0,L])$
be a solution of the equation \eqref{eq:sdw:sec5}
with initial data
$u(0,x) \in C^4([0,L])$
and 
$\partial_t u(0,x) \in C^3([0,L])$.
Let 
$\vect{U}^{(n)}$
be the solution of the difference scheme \eqref{eq:scheme:semi:sec4}
with the initial data
$\vect{U}^{(0)} = \vect{u}^{(0)}$
and
$\vect{U}^{(1)} = \vect{u}^{(1)}$. 
Denote the $L^{\infty}$-bound of solutions by $C_1$, namely, $C_1 := \max_{n=0,1,\ldots,N} \max \{ \| \vect{U}^{(n)} \|_{\infty}, \| \vect{u}^{(n)} \|_{\infty} \}$. 
Let us fix the constant $B$ satisfying $B < 1/(1 + C_{\tilde{F}''}(C_1))$. 
Then for $\Delta t < B$ we have 
\begin{align}
    \max_{n=1,2,\ldots,N} \left( \| \delta_n^- \vect{e}^{(n)} \|_2
    + \| \vect{e}^{(n)} \|_{H^1}
    + |\delta_n^- e_0^{(n)}|
    + |\delta_n^- e_K^{(n)}|
    \right) 
    \le C ((\Delta x)^2 + (\Delta t)^2), 
\end{align}
where the constant
$C$ is independent of
$\Delta x$ and $\Delta t$.
\end{theorem}
\begin{proof}
Let
$\vect{\zeta}^{(n)}=\{ \zeta_k^{(n)} \}_{k=0}^{K}$
be defined by
\begin{align}
    \zeta_k^{(n)}
    &:=
    -( \delta_n^{\langle 2 \rangle} u_k^{(n)} - \partial_t^2 u_k^{(n)} )
    + \left\{ \delta_k^{\langle 2 \rangle} 
    \left( \frac{u_k^{(n+1)} + u_k^{(n-1)}}{2} \right) - \partial_x^2 u_k^{(n)} \right\} \\
    &\qquad - \left\{  \frac{d \widetilde{F}}{d(u_k^{(n+1)},u_k^{(n-1)})} - \tilde{F}'(u_k^{(n)}) \right\}
 \end{align}
and let
\begin{align}
    \zeta_{b,0}^{(n)}
    &:=
    -( \delta_n^{\langle 2\rangle} u_0^{(n)} - \partial_t^2 u_0^{(n)} )
    + \left\{ \delta_k^{\langle 1\rangle}
    \left( \frac{u_0^{(n+1)} + u_0^{(n-1)}}{2} \right) - \partial_x^2 u_0^{(n)} \right\},\\
    \zeta_{b,K}^{(n)}
    &:=
    -( \delta_n^{\langle 2\rangle} u_K^{(n)} - \partial_t^2 u_K^{(n)} )
    - \left\{ \delta_k^{\langle 1\rangle}
    \left( \frac{u_K^{(n+1)} + u_K^{(n-1)}}{2} \right) - \partial_x^2 u_K^{(n)} \right\}.
\end{align}
We also define
$\vect{\eta}^{(n)} = \{ \eta_k^{(n)} \}_{k=0}^K$
by
\begin{align}
    \eta_k^{(n)}
    := \frac{d \widetilde{F}}{d(U_k^{(n+1)}, U_k^{(n-1)})}
    - \frac{d \widetilde{F}}{d(u_k^{(n+1)}, u_k^{(n-1)})}.
\end{align}
Then, by \eqref{eq:sdw:sec5} and \eqref{eq:scheme:semi:sec4},
we have
\begin{align}\label{eq:scheme:e}
    \left\{
    \begin{alignedat}{1}
    &\delta_n^{\langle 2\rangle} e_k^{(n)}
    = \delta_k^{\langle 2 \rangle}
    \left(
        \frac{e_k^{(n+1)}+e_k^{(n-1)}}{2}
    \right)
    - \eta_k^{(n)} + \zeta_k^{(n)}
        ,\\
    &\delta_n^{\langle 2\rangle} e_0^{(n)}
        - \delta_k^{\langle 1 \rangle}
        \left(
            \frac{e_0^{(n+1)}+e_0^{(n-1)}}{2}
        \right)
        = \zeta_{b,0}^{(n)},\\
    &\delta_n^{\langle 2\rangle} e_K^{(n)}
       + \delta_k^{\langle 1 \rangle}
        \left(
            \frac{e_K^{(n+1)}+e_K^{(n-1)}}{2}
        \right)
        = \zeta_{b,K}^{(n)},
    \end{alignedat}
    \right.
\end{align}
where
$n=1,\ldots,N-1$ and $k=0,\ldots,K$.
Multiplying $\eqref{eq:scheme:e}_1$ by
$\delta_n^{\langle 1 \rangle} e_k^{(n)}$
and summing the resulting equations 
following the trapezoidal rule with the help of
$\eqref{eq:scheme:e}_2$ and $\eqref{eq:scheme:e}_3$,
we obtain 
\begin{align}\label{eq:e:energy:sec5}
    &\frac{1}{\Delta t} \left( \| \delta_n^- \vect{e}^{(n+1)} \|_2^2 - \| \delta_n^- \vect{e}^{(n)} \|_2^2 \right) \\
    &\quad
    + \frac{1}{\Delta t} \left( \frac{\| D \vect{e}^{(n+1)} \|^2 + \| D \vect{e}^{(n)} \|^2}{2} - \frac{\| D \vect{e}^{(n)} \|^2 + \| D \vect{e}^{(n-1)} \|^2}{2} \right)\\
    &\quad
    + \frac{1}{\Delta t}
    \left(
    | \delta_n^- e_0^{(n+1)} |^2 - |\delta_n^- e_0^{(n)} |^2
    + | \delta_n^- e_K^{(n+1)} |^2 - |\delta_n^- e_K^{(n)} |^2
    \right) \\
    &\leq (\| \vect{\eta}^{(n)} \|_2
    + \| \vect{\zeta}^{(n)} \|_2
    ) 
    \| \delta_n^{\langle 1 \rangle} \vect{e}^{(n)} \|_2 \\
    &\quad
    + |\zeta_{b,0}^{(n)}| | \delta_n^{\langle 1 \rangle} e_0^{(n)}|
    + |\zeta_{b,K}^{(n)}| | \delta_n^{\langle 1 \rangle} e_K^{(n)}|.
\end{align}
From an argument similar to the previous section, we see that 
\begin{align}\label{eq:e:L2:sec5}
\| \vect{e}^{(n+1)} \|_2^2 - \| \vect{e}^{(n-1)} \|_2^2
\leq
2 \Delta t  \| \delta_n^{\langle 1 \rangle}  \vect{e}^{(n)} \|_2 
(\| \vect{e}^{(n+1)} \|_2+ \| \vect{e}^{(n-1)} \|_2).
\end{align}
Now, we set the error energy
$\mathcal{E}^{(n)}$
by 
\begin{align}
    \mathcal{E}^{(n)}
    := \| \delta_n^- \vect{e}^{(n)} \|_2^2 + \frac{\| \vect{e}^{(n+1)} \|_{H^1}^2 + \| \vect{e}^{(n)} \|_{H^1}^2}{2}
    + |\delta_n^- e_0^{(n)}|^2 + |\delta_n^- e_K^{(n)}|^2.
\end{align}
Then, adding
\eqref{eq:e:energy:sec5} and \eqref{eq:e:L2:sec5}
implies
\begin{align}
    \mathcal{E}^{(n+1)} - \mathcal{E}^{(n)} 
    &\leq \Delta t (\| \vect{\eta}^{(n)} \|_2 + \| \vect{\zeta}^{(n)} \|_2
    + \| \vect{e}^{(n+1)} \|_2+ \| \vect{e}^{(n-1)} \|_2)  \| \delta_n^{\langle 1 \rangle} \vect{e}^{(n)} \|_2 \\
    &\quad
    + |\zeta_{b,0}^{(n)}| | \delta_n^{\langle 1 \rangle} e_0^{(n)}|
    + |\zeta_{b,K}^{(n)}| | \delta_n^{\langle 1 \rangle} e_K^{(n)}|.
\end{align}
It follows from Lemma \ref{prop-f} (2) that 
\[
\| \vect{\eta}^{(n)} \|_2 \leq C_{\tilde{F}''}(C_1) (\| \vect{e}^{(n+1)} \|_2 + \| \vect{e}^{(n)} \|_2). 
\]
Observe that from the same argument as in \cite{5} 
\[
|\zeta_k^{(n)}|, |\zeta_{b,0}^{(n)}|, |\zeta_{b,K}^{(n)}| 
\leq C\left\{ (\Delta x)^2 + (\Delta t)^2 \right\}. 
\]
Then we see that for any $\varepsilon>0$
\begin{align*}
    \mathcal{E}^{(n+1)} - \mathcal{E}^{(n)} 
    &\leq \Delta t (1+  C_{\tilde{F}''}(C_1) ) (\mathcal{E}^{(n+1)} + \mathcal{E}^{(n)}) \\
    &\quad  + C \Delta t \left\{ (\Delta x)^2 + (\Delta t)^2 \right\} (\sqrt{\mathcal{E}^{(n+1)}} + \sqrt{\mathcal{E}^{(n)}} ) \\ 
&\leq \Delta t (1+  C_{\tilde{F}''}(C_1) + \varepsilon) (\mathcal{E}^{(n+1)} + \mathcal{E}^{(n)}) 
 + \frac{C}{\varepsilon} \Delta t \left\{ (\Delta x)^4 + (\Delta t)^4 \right\}. 
\end{align*}
If we denote $1+  C_{\tilde{F}''}(C_1) + \varepsilon$ by $A$, the above estimate is rewritten as 
\begin{align}\label{eq:error:energy:n+1:n}
    \mathcal{E}^{(n+1)} \leq \frac{1 + A \Delta t}{1 - A\Delta t} \mathcal{E}^{(n)} + \frac{C \Delta t}{1- A\Delta t} \left\{ (\Delta x)^4 + (\Delta t)^4 \right\}. 
\end{align}
Since $B< \frac{1}{1+C_{\tilde{F}''}(C_1) }$, we can fix $\varepsilon > 0$ satisfying $B< \frac{1}{1+ C_{\tilde{F}''}(C_1) + \varepsilon} < \frac{1}{1+C_{\tilde{F}''}(C_1) }$. 
For $\Delta t < B$, we observe that 
\[
\frac{1}{1 - A\Delta t} \leq \frac{1}{1 - AB}, 
\qquad \frac{1 + A \Delta t}{1 - A\Delta t} \leq 1 + \frac{2A}{1 -A B} \Delta t \leq \exp \left( \frac{2A}{1 -A B} \Delta t \right), 
\]
Noting them and applying the inequality
\eqref{eq:error:energy:n+1:n}
repeatedly,
for any
$\Delta t < B$
we have
\[
\mathcal{E}^{(n)} \leq \frac{CT}{1 -A B} \exp \left( \frac{2AT}{1 -A B} \right) \left\{ (\Delta x)^4 + (\Delta t)^4 \right\}, 
\]
which completes the proof. 
\end{proof}

\section{Numerical computations}
In this section, we show
numerical computations of
our difference schemes.
We set
$T=5, L=6$
and consider the case of semilinear wave equation with
the nonlinearity
$-u^3$,
that is, the original equations are
\begin{align}\label{eq:ndw:-u^3}
    \left\{ \begin{aligned}[2]
    &\frac{\partial^2 u}{\partial t^2} = \frac{\partial^2 u}{\partial x^2} - u^3,
    &&(t,x) \in (0,6)\times (0,6),\\
    &\frac{\partial^2 u}{\partial t^2}(t,0)
    - \frac{\partial u}{\partial x}(t,0) = 0,
    &&t\in (0,6),\\
    &\frac{\partial^2 u}{\partial t^2}(t,6)
    +\frac{\partial u}{\partial x}(t,6) = 0,
    &&t\in (0,6).
    \end{aligned}\right.
\end{align}
The corresponding difference scheme is given by
\eqref{eq:scheme:-u^3:sec3}.

In the following numerical simulation,
we take
$N=2000$, $K = 100$,
and three kinds of initial data:
\begin{itemize}
    \item[Case 1.]
    $u(0,x) = e^{-(x-L/2)^2}$
    and
    $\partial_t u(0,x) = 0$, 
    \item[Case 2.]
    $u(0,x) = e^{-4(x-L/3)^2} + e^{-4(x-2L/3)^2}$
    and
    $\partial_t u (0,x) = -4(x-L/3) e^{-4(x-L/3)^2}$, 
    \item[Case 3.]
    $u(0,x) = 5 e^{-4(x-L/3)^2} + e^{-4(x-2L/3)^2}$
    and
    $\partial_t u (0,x) = -4(x-L/3) e^{-4(x-L/3)^2}$. 
\end{itemize}
The behavior of solutions are given in Figure \ref{fig:1}. 
\begin{figure}[H]
\begin{center}
{\includegraphics[width=6cm]{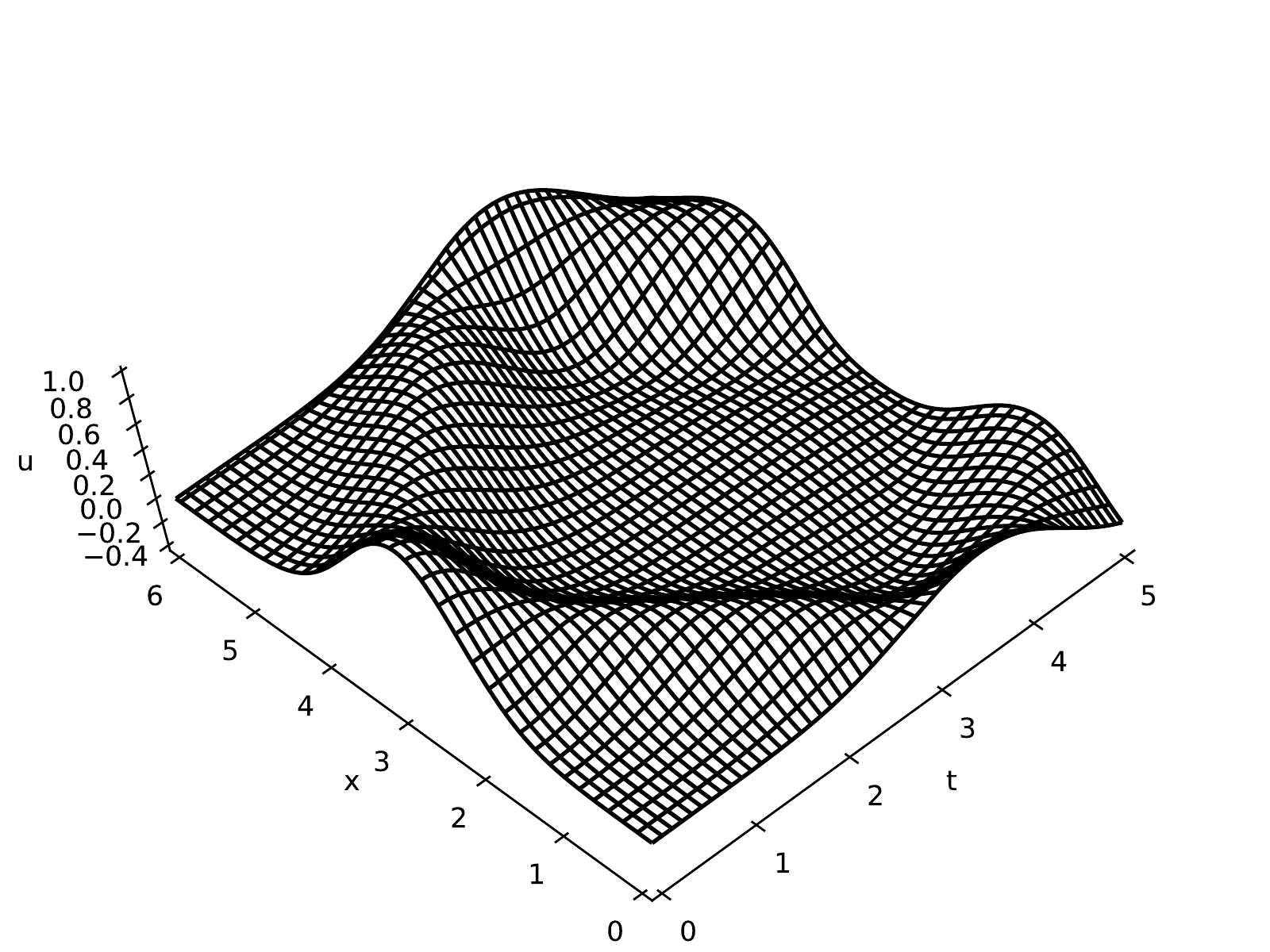} 
\includegraphics[width=6cm]{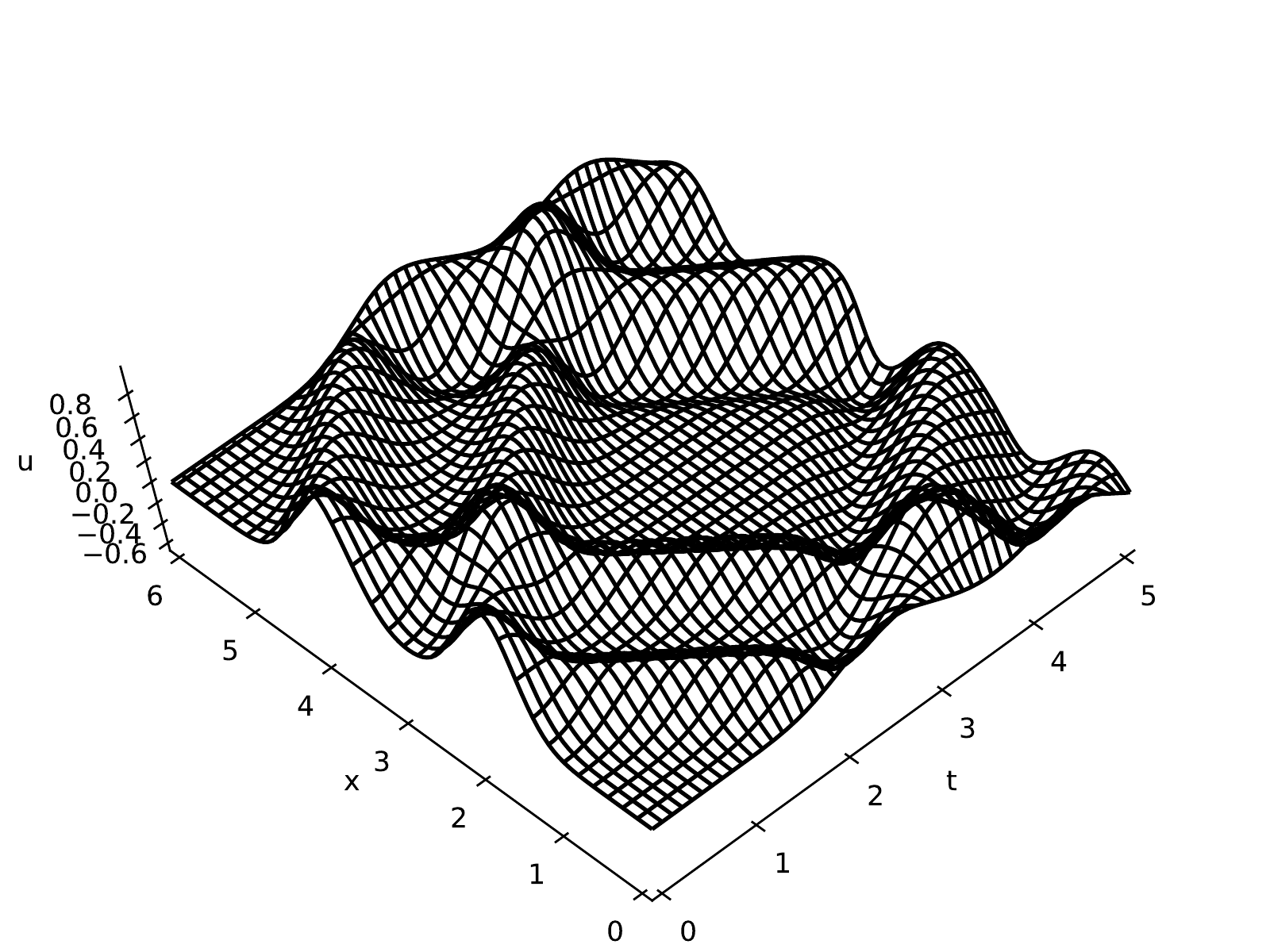}
\includegraphics[width=6cm]{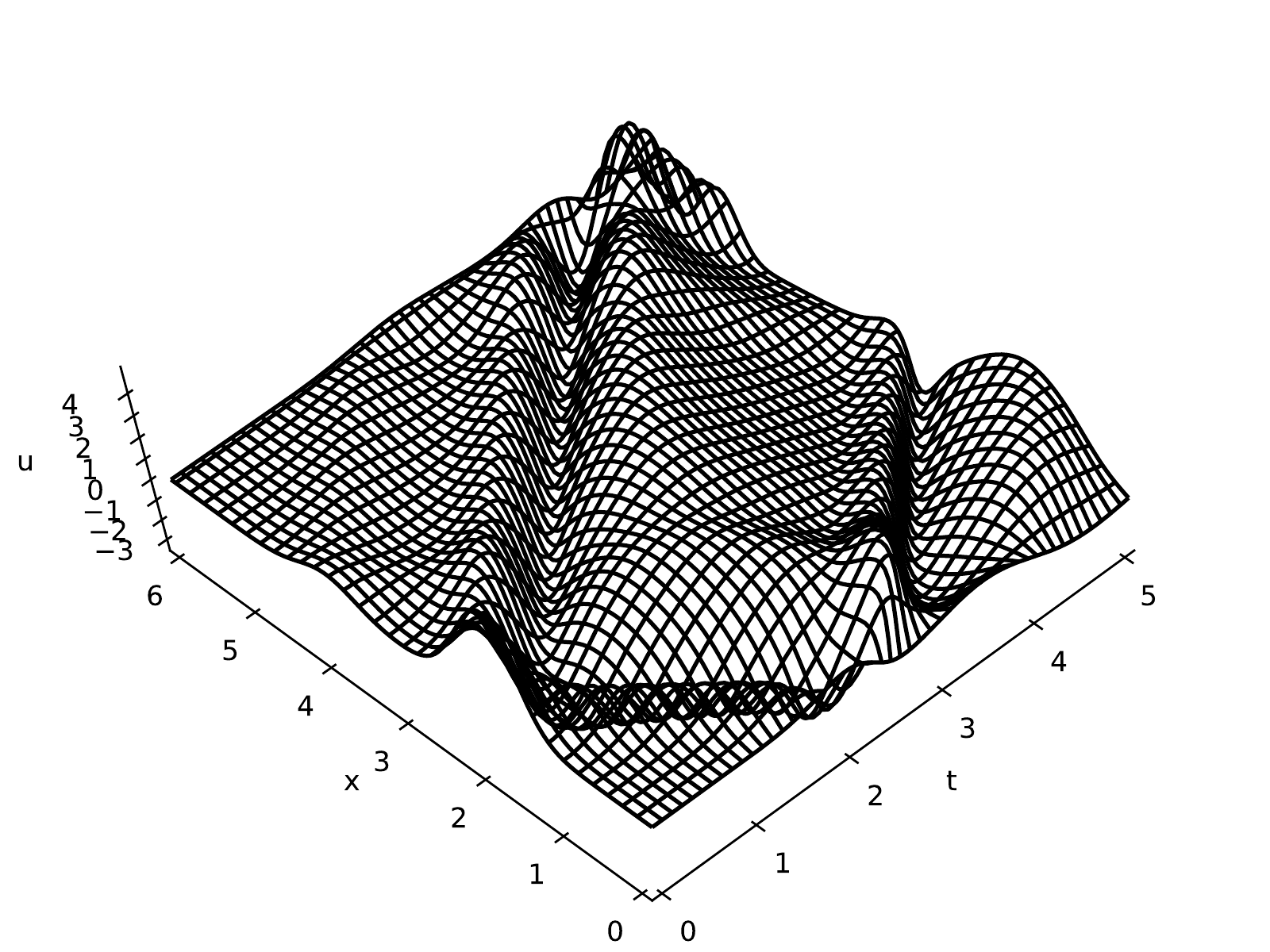}}
\caption{$U_k^{(n)}$ in Case 1 (upper left), Case 2 (upper right) and Case 3 (lower)}
    \label{fig:1} 
\end{center}
\end{figure}

The energy conservation for the solutions are confirmed from Figure \ref{fig:2}.  
\begin{figure}[H]
\begin{center}
\includegraphics[width=12cm]{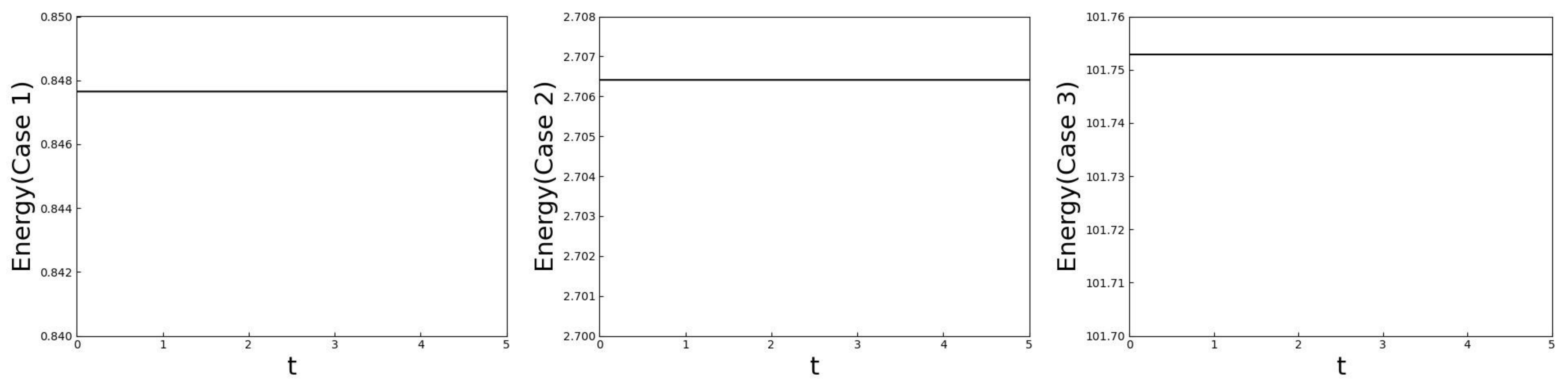} 
\caption{The conservation of the energy $J_{\mathrm{d}}(\vect{U}^{(n+1)}, \vect{U}^{(n)})$ for \eqref{eq:scheme:-u^3:sec3}}
    \label{fig:2} 
\end{center}
\end{figure}

In Figure \ref{fig:neuman},
we show a comparison with
the case of Neumann boundary condition. 
When we studied the parabolic type problem such as the Cahn-Hilliard equation or 
the Allen-Cahn equation, the behavior of solution with dynamic boundary condition near boundary tends to the 
one with the Neumann boundary condition, due to the parabolic property. 
On the other hand, we can observe a distinct behavior in the case of hyperbolic 
problem. 
\begin{figure}[H]
\begin{center}
\includegraphics[width=10cm]{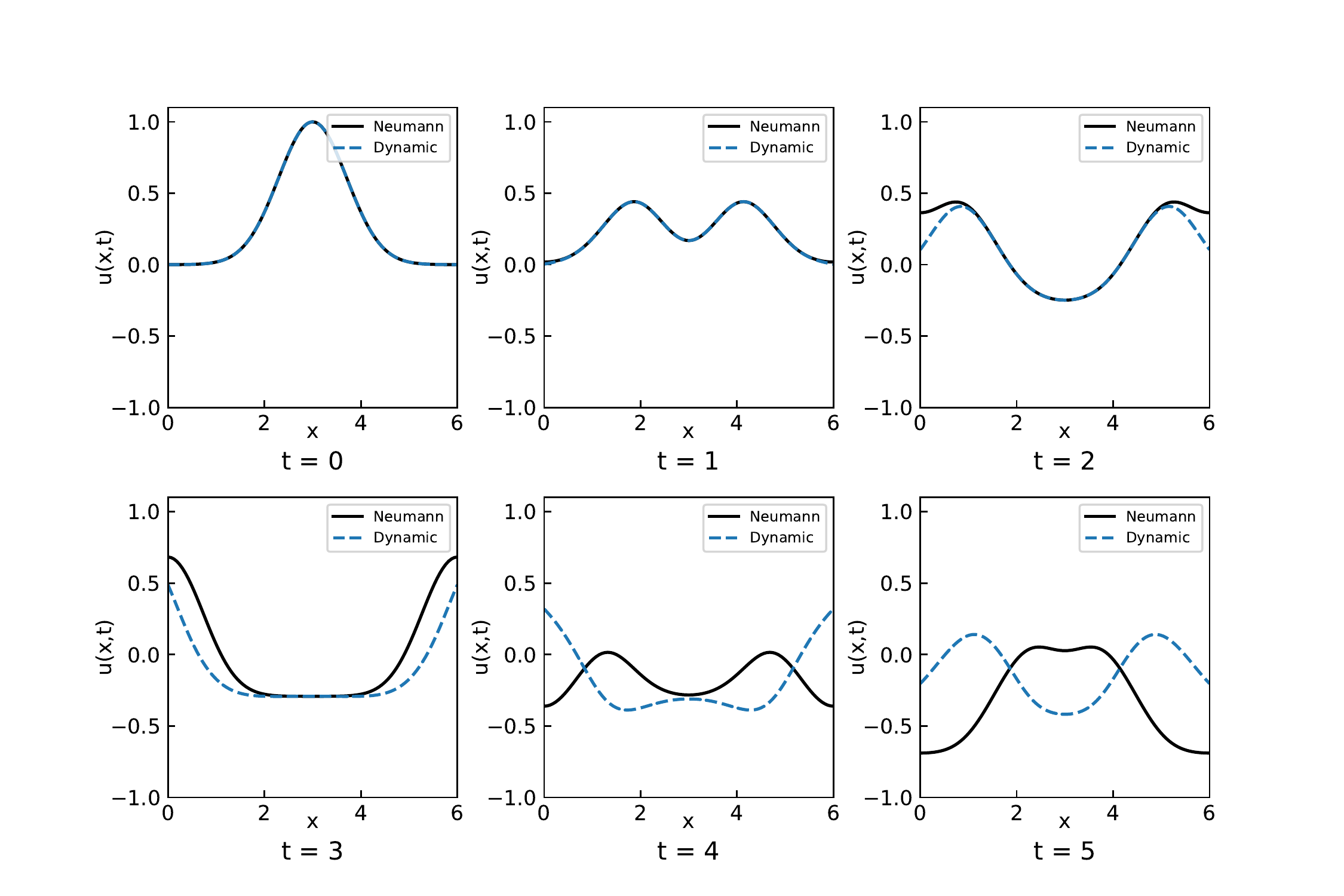}
\caption{Comparison with Neumann boundary condition for
the initial data
$u(0,x) = e^{-(x-L/2)^2}$
and $\partial_t u(0,x) = 0$ (Case 1)}
    \label{fig:neuman} 
\end{center}
\end{figure}

Lastly, we also show the numerical simulation for the sine-Gordon equation: 
\begin{align}\label{eq:ndw:sin(u)}
    \left\{ \begin{aligned}[2]
    &\frac{\partial^2 u}{\partial t^2} = \frac{\partial^2 u}{\partial x^2} - \sin(u),
    &&(t,x) \in (0,6)\times (0,6),\\
    &\frac{\partial^2 u}{\partial t^2}(t,0)
    - \frac{\partial u}{\partial x}(t,0) = 0,
    &&t\in (0,6),\\
    &\frac{\partial^2 u}{\partial t^2}(t,6)
    +\frac{\partial u}{\partial x}(t,6) = 0,
    &&t\in (0,6).
    \end{aligned}\right.
\end{align}
As mentioned in Section 3, the corresponding difference scheme is given by
\eqref{eq:scheme:sineGordon:sec3}.

The behavior of solutions are given in Figure \ref{fig:s1}. 
The simulation is demonstrated for the same initial conditions as 
the polynomial nonlinear case \eqref{eq:ndw:-u^3}. 
\begin{figure}[H]
\begin{center}
{\includegraphics[width=6cm]{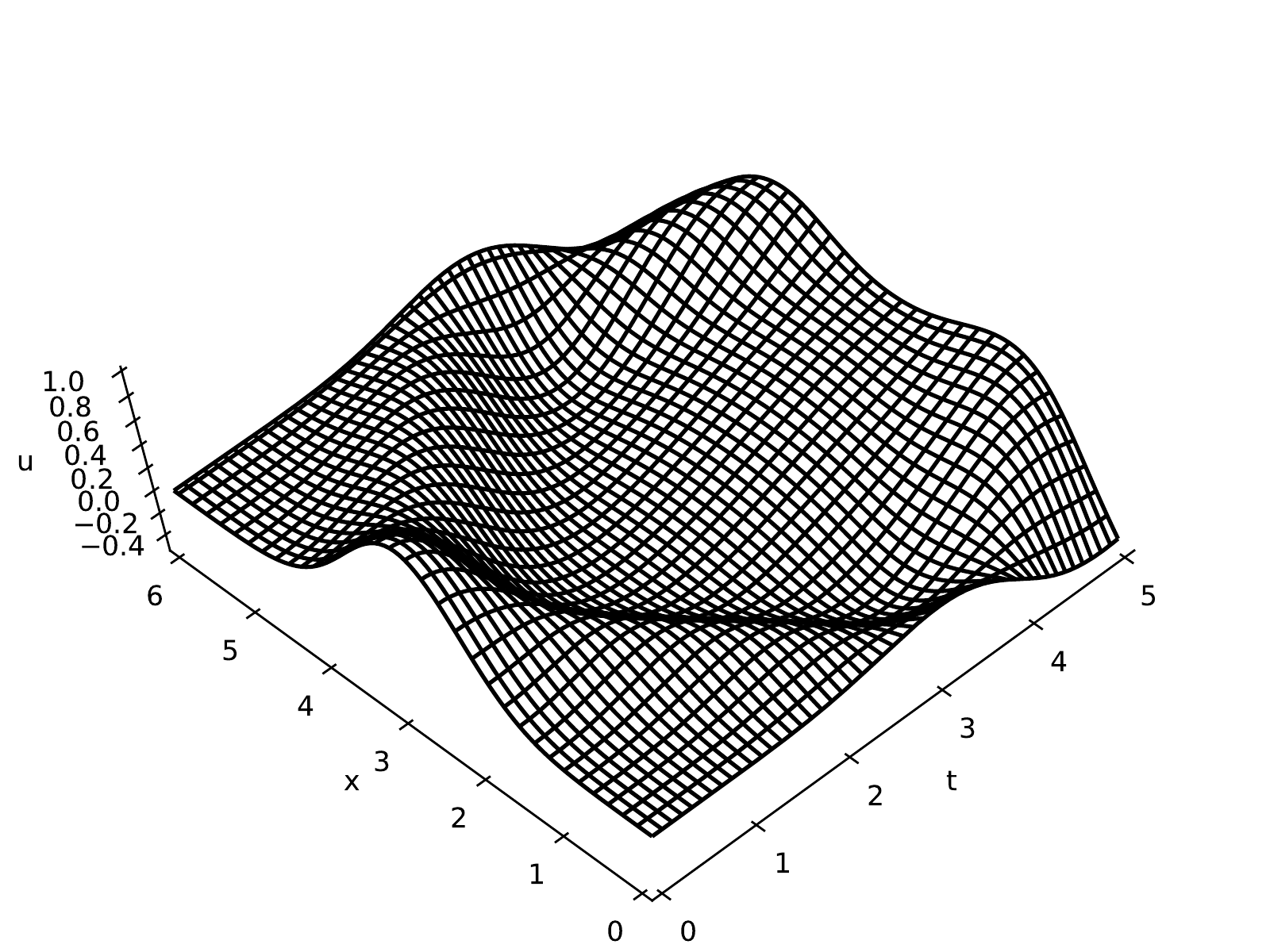} 
\includegraphics[width=6cm]{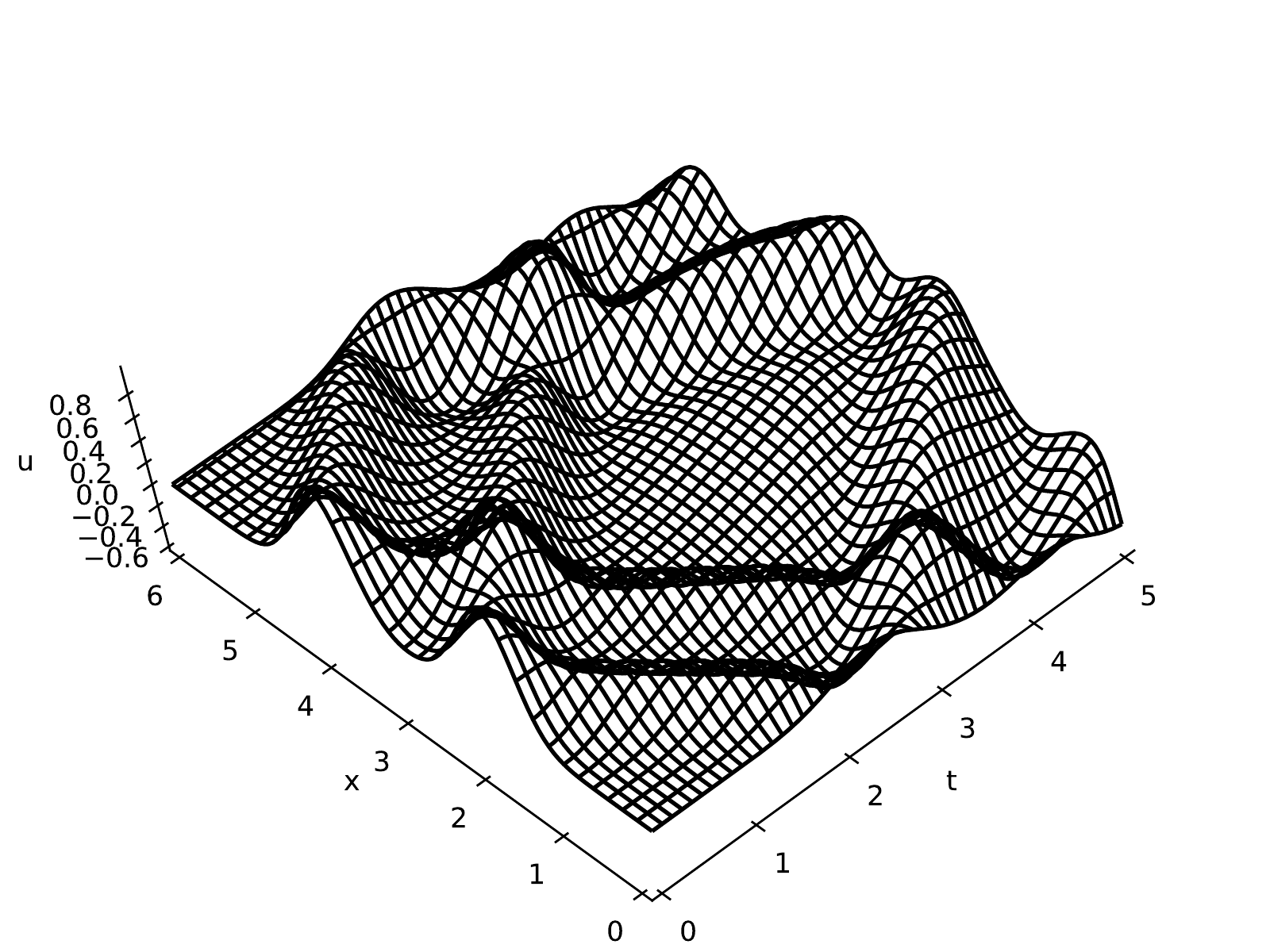}
\includegraphics[width=6cm]{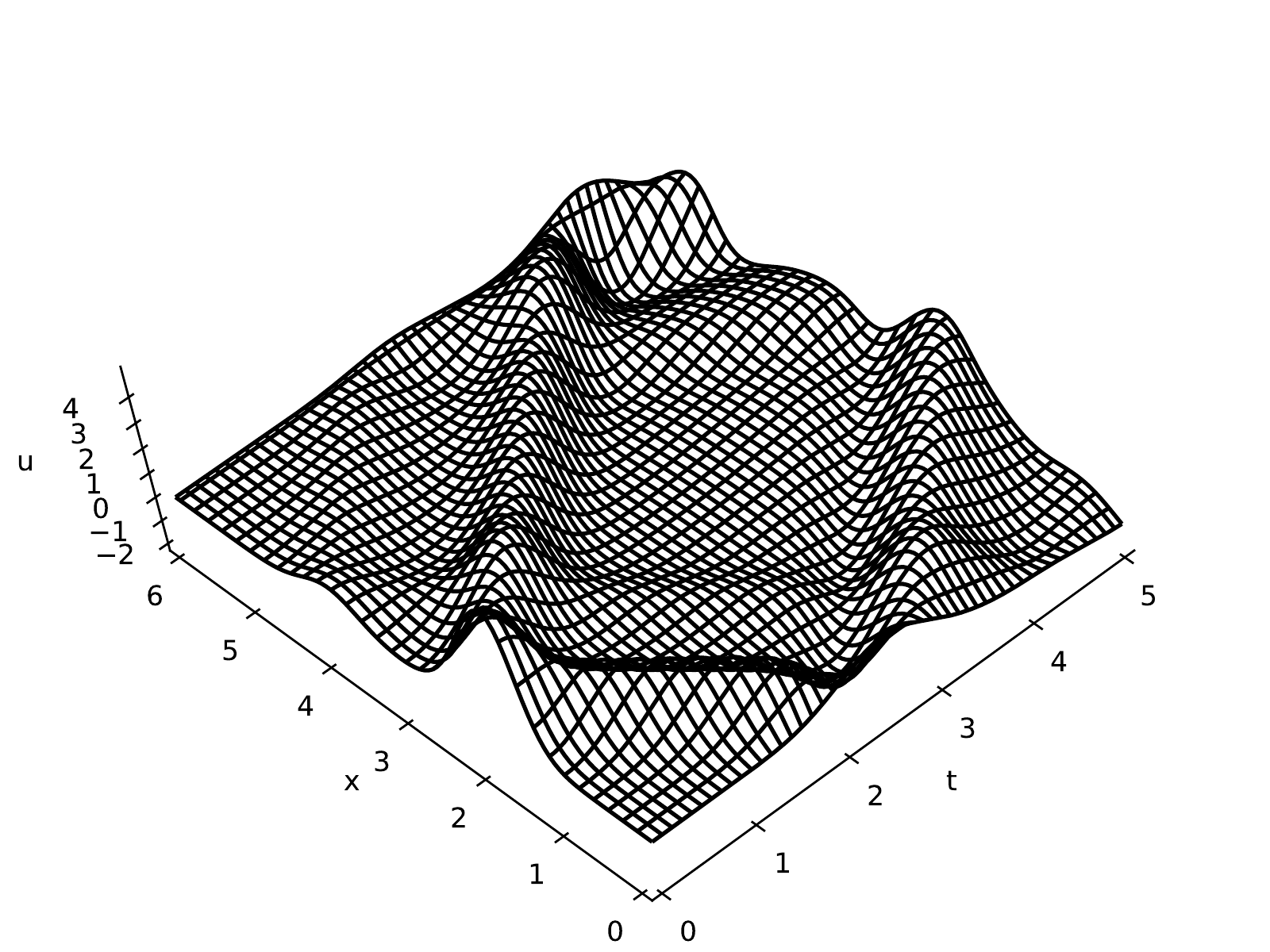}}
\caption{$U_k^{(n)}$ for \eqref{eq:scheme:sineGordon:sec3} in Case 1 (upper left), Case 2 (upper right) and Case 3 (lower)}
    \label{fig:s1} 
\end{center}
\end{figure}
The energy conservation for the solutions are confirmed from Figure \ref{fig:s2}.  
\begin{figure}[H]
\begin{center}
\includegraphics[width=12cm]{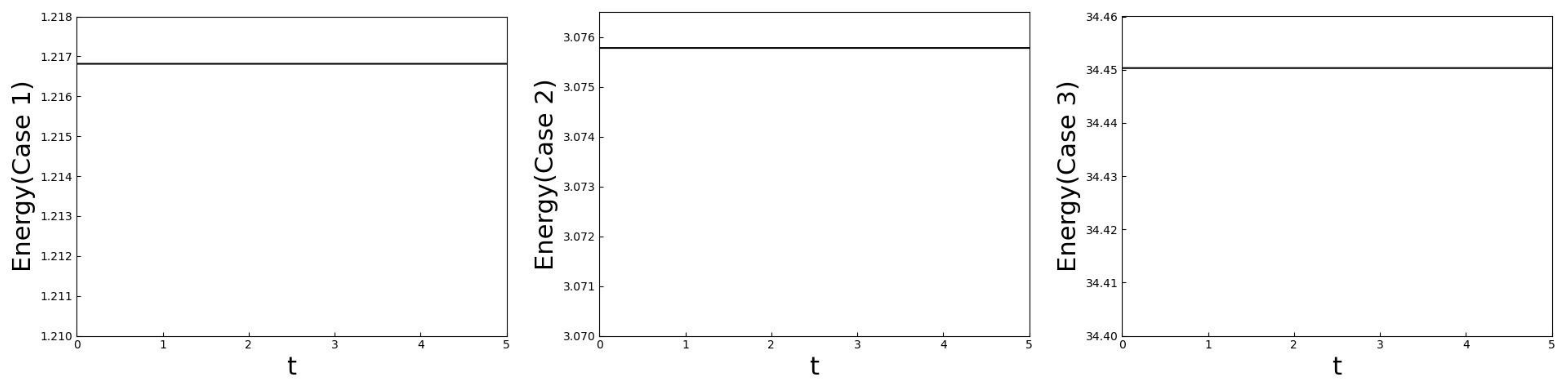} 
\caption{The conservation of the energy $J_{\mathrm{d}}(\vect{U}^{(n+1)}, \vect{U}^{(n)})$ for \eqref{eq:scheme:sineGordon:sec3}}
    \label{fig:s2} 
\end{center}
\end{figure}

\section*{Acknowledgments}
This work was written while
the first author was a master course student in
Ehime University.
This work was supported by JSPS KAKENHI Grant Numbers 
JP16K05234,
JP16K17625,
JP18H01132,
JP20K14346, 
JP20KK0308.


\end{document}